\documentclass[11pt]{article}

\usepackage[leqno]{amsmath}
\usepackage{amssymb,amsthm,bm,url,paralist,tabls}

\usepackage{graphicx}
\usepackage[font=small,labelfont=bf]{caption}

\theoremstyle{plain}      \newtheorem*{thm*}{Theorem}
                         \newtheorem{thm}{Theorem}[section]
                         \newtheorem{lemma}[thm]{Lemma}
                         \newtheorem{prop}[thm]{Proposition}
                         \newtheorem{cor}[thm]{Corollary}
\theoremstyle{remark}	  \newtheorem*{remark}{Remark}
\theoremstyle{definition} \newtheorem{example}{Example}%[section]

\newtheoremstyle{efronremark}% entire heading is optional argument []
{6pt}{6pt}{}{}{\itshape}{\quad}{ }{\thmnote{#3}}

\theoremstyle{efronremark}   \newtheorem*{eremark}{}

\date{}

\numberwithin{equation}{section}

\usepackage{varioref}
\labelformat{section}{Section~#1}
\labelformat{subsection}{Section~#1}
\labelformat{subsubsection}{Section~#1}
\labelformat{figure}{Figure~#1}
\labelformat{table}{Table~#1}
\labelformat{defin}{Definition~#1}

\usepackage[numbers,sort]{natbib}
\newcommand{\<}{\langle}  \renewcommand{\>}{\rangle}
\newcommand{\bbq}{\mathbb Q}
\newcommand{\real}{\mathbb R}
\newcommand{\calx}{\mathcal X}
\newcommand{\calp}{\mathcal P}
\newcommand{\barb}{\bar\beta}
\newcommand{\barf}{\bar{f}}
\newcommand{\barm}{\bar{M}}
\newcommand{\lam}{\lambda}
\newcommand{\lamp}{(\lambda)}
\newcommand{\laml}{\lambda_{J^c}}
\newcommand{\piqt}{\pi_{q,t}}
\newcommand{\piit}{\pi_{\infty,t}}
\newcommand{\dqt}{D_{q,t}}

\newcommand{\dal}{D_\alpha}
\newcommand{\plam}{P_\lam(x;q,t)}
\newcommand{\xrl}{X_\rho^\lam}
\newcommand{\toil}{\underset{\infty}{\overset{\ell}{\to}}}

\newcommand{\defeq}{\overset{\text{def}}{=}\ }

\DeclareMathOperator\dg{deg}

\begin{document}

\title{A probabilistic interpretation of the Macdonald polynomials}

\author{\textsc{Persi Diaconis}\footnote{Supported in part by NSF grant 0804324.}\ %
\footnote{Corresponding author: 390 Serra Mall, Stanford, CA 94305-4065. %\texttt{Persi.Diaconis@stanford.edu}
}
\\
\textit{Departments of}\\\textit{Mathematics and Statistics}\\\textit{Stanford University}\and
       \textsc{Arun Ram}\footnote{Supported in part by NSF grant 0353038, 
      and ARC grants DP0986774 and DP087995.}\\
\textit{Department of}\\\textit{Mathematics and Statistics}\\\textit{University of Melbourne}}

\maketitle

\begin{abstract}
 The two-parameter Macdonald polynomials are a central object of
 algebraic combinatorics and representation theory. We give a Markov
 chain on partitions of $k$ with eigenfunctions the coefficients of
 the Macdonald polynomials when expanded in the power sum
 polynomials. The Markov chain has stationary distribution a new
 two-parameter family of measures on partitions, the inverse of the
 Macdonald weight (rescaled). The uniform distribution on
 permutations and the Ewens sampling formula are special cases. The
 Markov chain is a version of the auxiliary variables algorithm of
 statistical physics. Properties of the Macdonald polynomials allow a
 sharp analysis of the running time. In natural cases, a bounded
 number of steps suffice for arbitrarily large $k$.

\begin{eremark}[Keywords:] Macdonald polynomials, random permutations, measures on partitions, auxiliary variables, Markov chain, rates of convergence
\ 
\end{eremark}

\begin{eremark}[AMS 2010 subject classifications:]
05E05 primary;  60J10 secondary.
\end{eremark}

\end{abstract}

\section{Introduction}\label{sec1}

The Macdonald polynomials $\plam$ are a widely studied family
of symmetric polynomials in variables $X=(x_1,x_2,\dots,x_n)$. Let
$\Lambda_n^k$ denote the vector space of homogeneous symmetric
polynomials of degree $k$ (with coefficients in $\bbq$). The Macdonald
inner product is determined by setting the inner product between power
sum symmetric functions $p_\lam$ as
\begin{equation*}
\<p_\lam,p_\mu\>=\delta_{\lam\mu}z_\lam(q,t),
\end{equation*}
with
\begin{equation}
z_\lam(q,t)=z_\lam\prod_i\left(\frac{1-q^{\lam_i}}{1-t^{\lam_i}}\right)\quad
\text{and}\quad z_\lam=\prod_ii^{a_i}a_i!\,,
\label{11}
\end{equation}
for $\lam$ a partition of $k$ with $a_i$ parts of size $i$.

For each $q,t$, as $\lam$ ranges over partitions of $k$, the
$\plam$ are an orthogonal basis for $\Lambda_n^k$. Special
values of $q,t$ give classical bases such as Schur functions ($q=t$),
Hall--Littlewood functions ($t=0$), and the Jack symmetric functions
(limit as $t\to1$ with $q^\alpha=t$). An enormous amount of
combinatorics, group theory, and algebraic geometry is coded into
these polynomials. A more careful description and literature review is
in \ref{sec2}.

The original definition of Macdonald constructs $\plam$ as the
eigenfunctions of a somewhat mysterious family of operators
$\dqt(z)$. This is used to develop their basic properties in
\cite{mac}. A main result of the present paper is that the Macdonald
polynomials can be understood through a natural Markov chain
$M(\lam,\lam')$ on the partitions of $k$. For $q,t>1$, this Markov
chain has stationary distribution
\begin{equation}
\piqt\lamp=\frac{Z}{z_\lam(q,t)}\qquad\text{with}\quad
Z=\frac{(q,q)_k}{(t,q)_k},\quad (x,y)_k=\prod_{i=0}^{k-1}\left(1-xy^i\right).
\label{12}
\end{equation}
Here $z_\lam(q,t)$ is the Macdonald weight \eqref{11} and $Z$ is a
normalizing constant. The coefficients of the Macdonald polynomials
expanded in the power sums give the eigenvectors of $M$, and there is a
simple formula for the eigenvalues.

Here is a brief description of $M$. From a current partition $\lam$,
choose some parts to delete: call these $\lam_J$. This leaves
$\laml=\lam\backslash\lam_J$. The choice of $\laml$ given $\lam$ is made with
probability
\begin{equation}
w_\lam(\laml)=\frac1{q^k-1}\prod_{i=1}^k\binom{a_i\lamp}{a_i(\laml)}(q^i-1)^{a_i\lamp-a_i(\laml)}.
\label{13}
\end{equation}
It is shown in \ref{sec2d} that for each $\lam$, $w_\lam(\cdot)$ is a
probability distribution with a simple-to-implement
interpretation. Having chosen $\laml$, choose a partition $\mu$ of
size $|\lam|-|\laml|$ with probability
\begin{equation}
\piit(\mu)=\frac{t}{t-1}\frac1{z_\mu}\prod\left(1-\frac1{t^i}\right)^{a_i(\mu)}.
\label{14}
\end{equation}
Adding $\mu$ to $\laml$ gives a final partition $\nu$. These two
steps define the Markov chain $M(\lam,\nu)$ with stationary
distribution $\piqt$. It will be shown to be a natural extension of
basic algorithms of statistical physics: the Swendsen--Wang and
auxiliary variables algorithms. Properties of the Macdonald
polynomials give a sharp analysis of the running time for $M$.

\ref{sec2} gives background on Macdonald polynomials (\ref{sec2a}),
Markov chains (\ref{sec2b}), and auxiliary variables algorithms
(\ref{sec2c}). The Markov chain $M$ is shown to be a special case of auxiliary
variables and
hence is reversible with $\piqt\lamp$ as stationary distribution.
\ref{sec2d} reviews some of the many different measures used on
partitions, showing that $w_\lam$ and $\piit$ above have simple
interpretations and efficient sampling algorithms. \ref{sec2d} also
presents simulations of the measure $\piqt\lamp$ using $M$. This gives
an understanding of $\piqt$; it also illustrates (numerically) that a
few steps of $M$ suffice for large $k$ while classical sampling
algorithms (rejection sampling or Metropolis) become impractical.

The main theorems are in \ref{sec3}. The Markov chain $M$ is
identified as one term of Macdonald operators $\dqt(z)$. The
coefficients of the Macdonald polynomials in the power sum basis
(suitably scaled) are shown to be the eigenfunctions of $M$ with a
simple formula for the eigenvalues. Needed values of the eigenvectors
are derived. A heuristic overview of the argument is given
(\ref{sec3b}), which may be read now for further motivation.

The main theorem is an extension of earlier work by \citeauthor{han}
\cite{han,pd86} giving a similar interpretation of the coefficients of
the family of Jack symmetric functions as eigenfunctions of a natural
Markov chain: the Metropolis algorithm on the symmetric group for
generating from the Ewens sampling formula. \ref{sec4} develops the
connection to the present study.

\ref{sec5} gives an analysis of the convergence of iterates of $M$ to
the stationary distibution $\piqt$ for natural values of $q$ and $t$.
Starting from $(k)$, it is shown that a bounded number of steps
suffice for arbitrary $k$. Starting from $1^k$, order $\log k$ steps
are necessary and sufficient for convergence.

\section{Background and examples}\label{sec2}

This section contains needed background on four topics: Macdonald
polynomials, Markov chains, auxiliary variables algorithms, and
measures on partitions and permutations. Each of these has a large
literature. We give basic definitions, needed formulae, and pointers
to literature. \ref{sec2c} shows that the Markov chain $M$ of the
introduction is a special case of the auxiliary variables algorithm.
\ref{sec2d} shows that the steps of the algorithm are easy to run, and
has numerical examples.

\subsection{Macdonald polynomials}\label{sec2a}

Let $\Lambda_n$ be the algebra of symmetric polynomials in $n$
variables (coefficients in $\bbq$). There are many useful bases of
$\Lambda_n$; the monomial $\{m_\lam\}$, power sum $\{p_\lam\}$,
elementary $\{e_\lam\}$, homogeneous $\{h_\lam\}$, and Schur functions
$\{s_\lam\}$ are bases whose change of basis formulae contain a lot of
basic combinatorics \cite[Chap.~7]{stanley}, \cite[Chap.~I]{mac}. More
esoteric bases such as the Hall--Littlewood functions $\{H_\lam(q)\}$,
zonal polynomials $\{Z_\lam\}$, and Jack symmetric functions
$\{J_\lam(\alpha)\}$ occur as the spherical functions of natural
homogeneous spaces \cite{mac}. In all cases, as $\lam$ runs over
partitions of $k$, the associated polynomials form a basis of the
vector space $\Lambda_n^k$: homogeneous symmetric polynomials of
degree $k$.

Macdonald introduced a two-parameter family of bases $\plam$ which,
specializing $q,t$ in various ways, gives essentially all the previous
bases. The Macdonald polynomials can be succinctly characterized by
using the inner product $\<p_\lam,p_\mu\>=\delta_{\lam\mu}z_\lam(q,t)$
with $z_\lam(q,t)$ from \eqref{11}. This is positive definite
\cite[VI~(4.7)]{mac} and there is a unique family of symmetric
functions $P(x;q,t)$ such that $\<P_\lam,P_\mu\>_{q,t}=0$ if
$\lam\neq\mu$ and $P_\lam=\sum_{\mu\leq\lam}u_{\lam\mu}m_\mu$ with
$u_{\lam\lam}=1$ \cite[VI~(4.7)]{mac}. The properties of $P_\lam$ are
developed by studying $P_\lam$ as the eigenfunctions of a family of
operators $\dqt(z)$ from $\Lambda_n$ to $\Lambda_n$.

Define an operator $T_{u,x_i}$ on polynomials by
$T_{u,x_i}f(x_1,\dots,x_n)=f(x_1,\dots,ux_i,\dots,x_n)$. Define
$\dqt(z)$ and $\dqt^r$ by
\begin{equation}
\dqt(z)=\sum_{r=0}^n\dqt^rz^r
=\frac1{a_\delta}\sum_{w\in S_n}\text{det}(w)\,x^{w\delta}\prod_{i=1}^n\left(1+zt^{(w\delta)_i}T_{q,x_i}\right) \ ,
\label{21}
\end{equation}
where $\delta=(n-1,n-2,\dots,0)$, $a_\delta$ is the Vandermonde
determinant and $x^\gamma = x_1^{\gamma_1}\cdots x_n^{\gamma_n}$ for $\gamma
=(\gamma_1,\ldots, \gamma_n)$. For any $r=0,1,\dots,n$,
\begin{equation}
\dqt^r=\sum_IA_I(x;t)\prod_{i\in I}T_{q,x_i} \ ,
\label{22}
\end{equation}
where the sum is over all $r$-element subsets $I$ of $\{1,2,\dots,n\}$
and
\begin{equation}
A_I(x;t)=\frac1{a_\delta}\left(\prod T_{t,x_i}\right)a_\delta
=t^{r(r-1)/2}\prod_{\substack{i\in I,\\j\notin I}}\frac{tx_i-x_j}{x_i-x_j}\qquad\text{\cite[VI~(3.4)$_r$]{mac}}.
\label{23}
\end{equation}

Macdonald \cite[VI~(4.15)]{mac} shows that the Macdonald polynomials
are eigenfunctions of $\dqt(z)$:
\begin{equation}
\dqt(z)\plam=\prod_{i=1}^n\left(1+zq^{\lam_i}t^{n-i}\right)\plam.
\label{24}
\end{equation}
This implies that the operators $\dqt^r$ commute, and have the
$P_\lam$ as eigenfunctions with eigenvalues the $r$th elementary
symmetric function in $\{q^{\lam_i}t^{n-i}\}$. We will use $\dqt^1$ in
our work below. The $\dqt^r$ are self-adjoint in the Macdonald inner
product $\<\dqt^rf,g\>=\<f,\dqt^rg\>$. This will translate into
having $\piqt$ as stationary distribution.

The Macdonald polynomials may be expanded in the power sums
\cite[VI~(8.19)]{mac},
\begin{equation}
\plam=\frac1{c_\lam(q,t)}\sum_\rho\left[z_\rho^{-1}\prod_i(1-t^{\rho_i})\xrl(q,t)\right]p_\rho(x)
\label{25}
\end{equation}
with \cite[VI~(8.1)]{mac}
$c_\lam(q,t)=\prod_{s\in\lam}(1-q^{a(s)}t^{l(s)+1})$ where the product
is over the boxes in the shape of $\lam$, $a(s)$ the arm length and
$l(s)$ the leg length of box $s$.
The $\xrl(q,t)$ are closely related to the two-parameter Kostka
numbers $K_{\mu\lam}(q,t)$ via \cite[VI~(8.20)]{mac},
\begin{equation}
\xrl(q,t)=\sum_\mu\chi_\rho^\lam K_{\mu\lam}(q,t),\qquad K_{\mu\lam}(q,t)=\sum_\rho z_\rho^{-1}\chi_\rho^\mu \xrl(q,t)
\label{26}
\end{equation}
with $\chi_\rho^\lam$ the characters of the symmetric group for the
$\lam$th representation at the $\rho$th conjugacy class. These
$K_{\mu\lam}(q,t)$ have been a central object of study in algebraic
combinatorics \cite{assaf}, \cite{gr}, \cite{gordon},
\cite{hhl,hhl1,hhl2}, \cite{hai}. The main result of \ref{sec3} shows
that $\xrl(q,t)\prod_i(1-q^{\rho_i})$ are the eigenfunctions of the
Markov chain $M$.

The Macdonald polynomials used here are associated to the root system
$A_n$. Macdonald \cite{mac2} has defined analogous functions for the
other root systems using similar operators. In a major step forward,
\citeauthor{cher} \cite{cher} gives an independent development in all
types, using the double affine Hecke algebra. See \cite{mac,mac03} for
a comprehensive treatment. Using this language, \citeauthor{ry}
\cite{ry} give a ``formula'' for the Macdonald polynomials in general
type. In general type the double affine Hecke is a powerful tool for
understanding actions.  We believe that our Markov chain can be
developed in general type if a suitable analogue of the power sum
basis is established.

\subsection{Markov chains}\label{sec2b}

Let $\calx$ be a finite set. A Markov chain on $\calx$ may be
specified by a matrix $M(x,y)\geq0,\ \sum_yM(x,y)=1$. The
interpretation being that $M(x,y)$ is the chance of moving from $x$ to
$y$ in one step. Then $M^2(x,y)=\sum_zM(x,z)M(z,y)$ is the chance of
moving from $x$ to $y$ in two steps, and $M^\ell(x,y)$ is the chance of 
moving from $x$ to $y$ in $\ell$
steps. Under mild conditions, always met in our examples, there is a
unique stationary distribution $\pi(x)\geq0,\ \sum_x\pi(x)=1$. This
satisfies $\sum_x\pi(x)M(x,y)=\pi(y)$. Hence, the (row) vector $\pi$
is a left eigenvector of $M$ with eigenvalue 1. Probabilistically,
picking $x$ from $\pi$ and taking one further step in the chain leads
to the chance $\pi(y)$ of being at $y$.

All of the Markov chains used here are reversible, satisfying the
detailed balance condition $\pi(x)M(x,y)=\pi(y)M(y,x)$, for all $x,y$
in $\calx$. Set $L^2(\calx)$ to be $\{f:\calx\to\real\}$ with
$(f_1,f_2)=\sum_x\pi(x)f_1(x)f_2(x)$. Then $M$ acts as a contraction
on $L^2(\calx)$ by $Mf(x)=\sum_yM(x,y)f(y)$. Reversibility is
equivalent to $M$ being self-adjoint. In this case, there is an
orthogonal basis of (right) eigenfunctions $f_i$ and real eigenvalues
$\beta_i,\ 1=\beta_0\geq\beta_1\geq\dots\geq\beta_{|\calx|-1}\geq-1$ with
$Mf_i=\beta_if_i$. For reversible chains, if $f_i(x)$ is a left
eigenvector, then $f_i(x)/\pi(x)$ is a right eigenvector with the same
eigenvalue.

A basic theorem of Markov chain theory shows that
$M_x^\ell(y)=M^\ell(x,y)\toil\pi(y)$. (Again, there are mild conditions, met in our
examples.) The distance to stationarity can be measured in $L^1$ by
the total variation distance:
\begin{equation}
\left\|M_x^\ell-\pi\right\|_{\text{TV}}
=\max_{A\subseteq\calx}\left|M^\ell(x,A)-\pi(A)\right|
=\hbox{$\frac12$}\sum_y\left|M^\ell(x,y)-\pi(y)\right|.
\label{27}
\end{equation}
Distance is measured in $L^2$ by the chi-squared distance:
\begin{equation}
\left\|M_x^\ell-\pi\right\|_2^2=\sum_y\frac{\left(M^\ell(x,y)-\pi(y)\right)^2}{\pi(y)}
=\sum_{i=1}^{|\calx|-1}\bar f_i^2(x)\beta_i^{2\ell},
\label{28}
\end{equation}
where $\barf_i$ is the eigenvector $f_i$, normalized to have $L^2$-norm $1$.
The Cauchy--Schwarz inequality shows
\begin{equation}
4\left\|M_x^\ell-\pi\right\|^2_{\text{TV}}\leq\left\|M_x^\ell-\pi\right\|_2^2.
\label{29}
\end{equation}
Using these bounds calls for getting one's hands on eigenvalues and
eigenvectors. This can be hard work, but has been done in many cases.
A central question is this: given $M,\ \epsilon>0$, and a starting
state $x$, how large must $\ell$ be so that
$\|M_x^\ell-\pi\|_{\text{TV}}<\epsilon$?

Background on the quantitative study of rates of convergence of Markov
chains is treated in the textbook of \citeauthor{bremaud}
\cite{bremaud}. The identities and inequalities that appear above are
derived in the very useful treatment by \citeauthor{saloff}
\cite{saloff}. He shows how tools of analysis can be brought to
bear. The recent monograph of \citeauthor*{levin} \cite{levin} is
readable by non-specialists and covers both analytic and probabilistic
techniques.

\subsection{Auxiliary variables}\label{sec2c}

This is a method of constructing a reversible Markov chain with $\pi$
as stationary distribution. It was invented by \citeauthor{edwards}
\cite{edwards} as an abstraction of the remarkable Swendsen--Wang
algorithm. The Swendsen--Wang algorithm was introduced as a superfast
method for simulating from the Ising and Potts models of statistical
mechanics.  It is a block-spin procedure which changes large pieces of
the current state. A good overview of such block spin algorithms is in
\cite{newman}. The abstraction to auxiliary variables is itself
equivalent to several other classes of widely used procedures, data
augmentation and the hit-and-run algorithm. For these connections and
much further literature, see \cite{pd168}.

To describe auxiliary variables, let $\pi(x)>0,\ \sum_x\pi(x)=1$ be a
probability distribution on a finite set $\calx$ Let $I$ be an
auxiliary index set. For each $x\in\calx$, let $w_x(i)$ be a
probability distribution on $I$ (the chance of moving to $i$). These
define a joint distribution $f(x,i)=\pi(x)w_x(i)$ and a marginal
distribution $m(i)=\sum_xf(x,i)$. Let $f(x|i)=f(x,i)/m(i)$ denote the
conditional distribution. The final ingredient needed is a Markov
matrix $M_i(x,y)$ with $f(x|i)$ as reversing measure
($f(x|i)M_i(x,y)=f(y|i)M_i(y,x)$ for all $x,y$). This allows for defining
\begin{equation}
M(x,y)=\sum_iw_x(i)M_i(x,y).
\label{210}
\end{equation}
The Markov chain $M$ has the following interpretation: from $x$, choose
$i\in I$ from $w_x(i)$ and then $y\in\calx$ from $M_i(x,y)$. The
resulting kernel is reversible with respect to $\pi$:
\begin{align*}
\pi(x)M(x,y)=&\sum_i\pi(x)w_x(i)M_i(x,y)=\sum_i\pi(y)w_y(i)M_i(y,x)\\
&= \pi(y)\sum_iw_y(i)M_i(y,x)=\pi(y)M(y,x).
\end{align*}

We now specialize things to $\calp_k$, the space of partitions of $k$.
Take $\calx=\calp_k,\ I=\cup_{i=1}^k\calp_i$. The stationary
distribution is as in \eqref{12}:
\begin{equation}
\pi\lamp=\piqt\lamp=\frac{Z}{z_\lam(q,t)}.
\label{211}
\end{equation}
From $\lam\in\calp_k$, the algorithm chooses some parts to delete,
call these $\lam_J$, leaving parts $\laml=\lam\backslash\lam_J$. Thus if
$\lam=322111$ and $\lam_J=31,\ \laml=2211$. We allow $\lam_J=\lam$ but
demand $\lam_J\neq\emptyset$. Clearly, $\lam$ and $\lam_J$ determine
$\laml$ and $(\lam,\laml)$ determine $\lam_J$. We let $\laml$ be the
auxiliary variable. The choice of $\laml$ given $\lam$ is made with
probability
\begin{equation}\begin{aligned}
w_\lam(\laml)&=\frac1{q^k-1}\prod_{i=1}^k\binom{a_i\lamp}{a_i(\laml)}
\left(q^i-1\right)^{a_i(\lam_J)}\\
           &=\frac1{q^k-1}\prod_{i=1}^k\binom{a_i\lamp}{a_i(\laml)}
           \left(q^i-1\right)^{a_i\lamp-a_i(\laml)}.
\end{aligned}\label{212}
\end{equation}
Thus, for $\lam=1^323^2;\ \lam_J=13,\ \laml=1^223;\
w_\lam(\laml)=\frac1{q^{11}-1}\binom{3}2\tbinom{1}1\tbinom{2}1(q-1)(q^2-1)^0(q^3-1)$. It
is shown in \ref{sec2d} below that $w_\lam(\laml)$ is a probability
distribution with a simple interpretation. Having chosen $\lam_J$ with
$0<|\lam_J|\leq k$, the algorithm chooses $\mu\vdash|\lam_J|$ with
probability $\piit(\mu)$ given in \eqref{14}. Adding these parts
to $\laml$ gives $\nu$. More carefully,
\begin{equation}
M_{\laml}(\lam,\nu)=\piit(\mu)
=\frac{t}{t-1}\frac1{z_{\mu}}\prod_i\left(1-\frac1{t^i}\right)^{a_i(\mu)}.
\label{213}
\end{equation}
Here it is assumed that $\laml$ is a part of both $\lam$ and $\nu$; the
kernel $M_{\laml}(\lam,\nu)$ is zero otherwise.

It is shown in \ref{sec2d} below that $M_{\laml}$ has a simple
interpretation which is easy to sample from. The joint density
$f(\lam,\laml)=\pi(\lambda)w_\lambda(\laml)$ is proportional to $f(\lam|\laml)$ and to
\begin{equation}
\frac{\prod_i(1-1/t^i)^{a_i\lamp}}{\prod_ii^{a_i\lamp}\left(a_i\lamp-a_i(\laml)\right)!}.
\label{214}
\end{equation}
The normalizing constant depends on $\laml$ but this is fixed in the
following. We must now check reversibility of $f(\lam|\laml),\
M_{\laml}(\lam,\nu)$. For this, compute
$f(\lam|\laml)M_{\laml}(\lam,\nu)$ (up to a constant depending on
$\laml$) as
\begin{equation*}
\frac{\prod_i(1-1/t^i)^{a_i\lamp+a_i(\nu)}}{\prod_ii^{a_i\lamp+a_i(\nu)}\left(a_i\lamp-a_i(\laml)\right)!\left(a_i(\nu)-a_i(\laml)\right)!}.
\end{equation*}
This is symmetric in $\lam,\nu$ and so equals
$f(\nu |\laml)M_{\laml}(\nu,\lam)$.  This proves the following:
\begin{prop}
 With definitions \eqref{211}--\eqref{214}, the kernel on $\calp_k$,
\begin{equation*}
M(\lam,\nu)=\sum_{\laml}w_\lam(\laml)M_{\laml}(\lam,\nu)
\end{equation*}
generates a reversible Markov chain with $\piqt\lamp$ as stationary
distribution.
\label{prop1}
\end{prop}
\begin{example}
 With $k=2$, let
\begin{equation*}
\piqt(2)=\frac{Z}2\frac{(t^2-1)}{(q^2-1)},\
 \piqt(1^2)=\frac{Z}2\left(\frac{t-1}{q-1}\right)^2\quad\text{for}\quad Z=\frac{(1-q)(1-q^2)}{(1-t)(1-tq)}.
\end{equation*}
From the definitions, with rows and columns labeled (2), $1^2$, the
transition matrix is
\begin{equation}
M=\begin{pmatrix}\dfrac12\left(1+\dfrac1{t}\right)&\dfrac12\left(1-\dfrac1{t}\right)\\[.2in]
\dfrac{q-1}{q+1}\dfrac12\left(1+\dfrac1{t}\right)&\dfrac{4t+(q-1)(t-1)}{2(q+1)t}\end{pmatrix}=
\frac1{2t}\begin{pmatrix}t+1&t-1\\[.1in]\dfrac{(q-1)(t+1)}{q+1}&\dfrac{4t+(q-1)(t-1)}{q+1}\end{pmatrix}.
\label{215}
\end{equation}
In this $k=2$ example, it is straightforward to check that $\piqt$ sums to $1$, the rows of $M$ sum to
$1$, and that $\piqt\lamp M(\lam,\nu)$ $=\piqt(\nu)M(\nu,\lam)$.
\end{example}

\subsection{Measures on partitions and permutations}\label{sec2d}

The measure $\piqt$ of \eqref{12} has familiar specializations: to the
distribution of conjugacy classes of a uniform permutation ($q=t$), and
the Ewens sampling  measure ($q^\alpha=t\to0$). After recalling these,
the measures $w_{\laml}(\cdot)$ and $M_{\laml}(\lam,\cdot)$ used in
the auxiliary variables algorithm are treated. Finally, there is a
brief review of the many other, nonuniform distributions used on
partitions $\calp_k$ and permutations $S_k$ . Along the way, many
results on the ``shape'' of a typical partition drawn from $\piqt$
appear.

\subsubsection{Uniform permutations ($q=t$)}\label{sub2d1}

If $\sigma$ is chosen uniformly on $S_k$, the chance that the cycle
type of $\sigma$ is $\lam$ is $1/z_\lam=\pi_{q,q}\lamp$. There is a
healthy literature on the structure of random permutations (number of
fixed points, cycles of length $i$, number of cycles, longest and
shortest cycles, order, \textellipsis). This is reviewed in
\cite{fulman,olshanski}, which also contain extensions to the
distribution of conjugacy classes of finite groups of Lie type.

One natural appearance of the measure $1/z_\lam$ comes from the
\textit{coagulation/fragmentation} process. This is a Markov chain on
partitions of $k$ introduced by chemists and physicists to study clump
sizes. Two parts are chosen with probability proportional to their
size. If different parts are chosen, they are combined. If the same
part is chosen twice, it is split uniformly into two parts. This
Markov chain has stationary distribution $1/z_\lam$. See \cite{aldous}
for a review of a surprisingly large literature and \cite{pd27} for
recent developments. These authors note that the
coagulation/fragmentation process is the random transpositions walk,
viewed on conjugacy classes. Using the Metropolis algorithm (as in
\ref{sub2d6} below) gives a similar process with stationary
distribution $\piqt$.

Algorithmically, a fast way to pick $\lam$ with probability $1/z_\lam$
is by \textit{uniform stick-breaking}: Pick $U_1\in\{1,\ldots,k\}$
uniformly. Pick $U_2\in\{1,\ldots, k-U_1\}$ uniformly. Continue until the first
time $T$ that the uniform choice equals its maximum value. The
partition with parts $U_1,U_2,\dots,U_T$ equals $\lam$ with
probability $1/z_\lam$.

\subsubsection{Ewens and Jack measures}\label{sub2d2}

Set $q=t^\alpha$ and let $t\to1$. Then $\piqt\lamp$ converges to
\begin{equation}
\pi_\alpha\lamp=\frac{Z}{z_\lam}\alpha^{-\ell\lamp},\quad
Z=\frac{\alpha^kk!}{\prod_{i=1}^{k-1}(i\alpha+1)},\quad \ell\lamp\text{ the number of parts of }\lam.
\label{216}
\end{equation}
In population genetics, setting $\alpha=1/\theta$, with $\theta>0$ a
``fitness parameter,'' this measure is called the \textit{Ewens
 sampling formula}. It has myriad practical appearances through its
connection with Kingman's coalescent process, and has generated a
large enumerative literature in the combinatorics and probability
community \cite{arratia,hoppe,pitman}. It also makes numerous
appearances in the statistics literature through its occurrence in
non-parametric Bayesian statistics via Dirichlet random measures and
the Dubins--Pitman Chinese restaurant process
\cite{ghosh03}, \cite[sec. 3.1]{pitman}.

Algorithmically, a fast way to pick $\lam$ with probability
$\pi_{1/\theta}\lamp$ is by the Chinese restaurant
construction. Picture a collection of circular tables. Person 1 sits
at the first table. Successive people sit sequentially, by choosing to
sit to the right of a (uniformly chosen) previously seated person
(probability $\theta$) or at a new table (probability
$1-\theta$). When $k$ people have been seated, this generates the
cycles of a random permutation with probability $\pi_{1/\theta}$. It
would be nice to have a similar construction for the measures $\piqt$.

The Macdonald polynomials associated to this weight function are
called the \textit{Jack symmetric functions} \cite[VI Sect.~1]{mac}.
Hanlon \cite{han,pd86} uses properties of Jack polynomials to
diagonalize a related Markov chain; see \ref{sec4}.  When
$\alpha=1/2$, the Jack polynomials become the zonal-spherical
functions of $GL_n/O_n$. Here, an analysis closely related to the
present paper is carried out for a natural Markov chain on
perfect matchings and phylogenetic trees \cite[Chap.~X]{ceccher},
\cite{pd38}.

\subsubsection{The measure $w_\lam$}\label{sub2d3}

Fix $\lam\vdash k$ with $\ell$ parts and $q>1$. Define, for
$J\subseteq \{1,\ldots, \ell\},\ J\neq\emptyset$,
\begin{equation}
w_\lam(J)=\frac1{q^k-1}\prod_{i\in J}\left(q^{\lam_i}-1\right).
\label{217}
\end{equation}
The auxiliary variables algorithm for sampling from $\piqt$ involves
sampling from $w_\lam(J)$, and setting $\lam_J=\{\lam_i:i\in J\}$ (see
\eqref{13} and \eqref{212}). The measure $w_\lam(J)$ has the following interpretation,
which leads to a useful sampling algorithm: Consider $k$ places
divided into blocks of length $\lam_i$:
\begin{equation*}
\underbrace{- - \cdots -}_{\lam_1}\underbrace{- - \cdots -}_{\lam_2}
\cdot\underbrace{- - \cdots -}_{\lam_l} \ ,
\qquad
\lam_1+\cdots+\lam_l=k.
\end{equation*}
Flip a $1/q$ coin for each place. Let, for $1\leq i\leq \ell$,
\begin{equation}
X_i=\begin{cases}1
&\text{if the $i$th block is \textit{not} all ones,}\\
0 &\text{otherwise.}\end{cases}
\label{218}
\end{equation}
Thus $P(X_i=1)=1-1/q^{\lam_i}$. Let $J=\{i:X_i=1\}$. So
$P\{J=\emptyset\}=1-1/q^k$ and
\begin{equation}
P\{J|J\neq\emptyset\}
=\frac1{1-\frac1{q^k}}\prod_{i\in J}\left(1-\frac1{q^{\lam_i}}\right)\prod_{j\in J^c}\frac1{q^{\lam_j}}=w_\lam(J).
\label{219}
\end{equation}
This makes it clear that summing $w_\lam(J)$ over all non-empty
subsets of $\{1,\ldots, \ell\}$ gives 1. 
%Of course, the measure induced on partitions
%of $\ell$ (instead of subsets of $J$) is given by \eqref{13}.

The simple rejection algorithm for sampling from $w_\lam$ is: Flip
coins as above. If $J\neq\emptyset$, output $\lam_J=\{\lam_i:i\in J\}$.
 If $J=\emptyset$, sample again. The chance of success is $1-1/q^k$.
 Thus, unless $q$ is very close to 1, this is an efficient algorithm.

 As $q$ tends to infinity, $w_\lam$ converges to point mass at
 $J=\{1,\ldots, k\}$. As $q$ tends to one, $w_\lam$ converges to the measure
 putting mass $\lam_i/k$ on $\{i\}$.

\subsubsection{The measure $\piit$}\label{sub2d4}

Generating from the kernel $M_{\laml}(\lam,\nu)$ of \eqref{213}
with $r=|\lam\backslash\laml|$, requires generating a partition in $\calp_r$
from
\begin{equation*}
\piit(\mu)=\left(\frac{t}{t-1}\right)
\frac1{z_\mu}\prod_i\left(1-\frac1{t^i}\right)^{a_i(\mu)}.
\end{equation*}
This measure has the following interpretation: Pick $\mu^{(1)}\vdash r$
with probability $1/z_{\mu^{(1)}}$. This may be done by picking a random
permutation in $S_r$ uniformly and reporting the cycle decomposition,
or by the uniform stick-breaking of \ref{sub2d1} above. For each part
$\mu_j^{(1)}$ of $\mu^{(1)}$, flip a $1/t$ coin $\mu_j^{(1)}$ times. 
If this comes up tails at least once, and this happens simultaneously for each $i$, 
set $\mu=\mu^{(1)}$. If some part of $\mu^{(1)}$ produces all heads, start
again and choose $\mu^{(2)}\vdash r$ with probability $1/z_{\mu^{(2)}}$ \textellipsis. 
The chance of failure is $1/t$, independent of $r$. Thus,
unless $t$ is close to 1, this gives a simple, useful algorithm.

The shape of a typical pick from $\piit$ is described in the following
section. When $t$ tends to infinity, the measure converges to
$1/z_\mu$. When $t$ tends to one, the measure converges to point mass
at the one part partition $(r)$.

\subsubsection{Multiplicative measures}\label{sub2d5}

For $\bm\eta=(\eta_1,\eta_2,\dots,\eta_k),\ \eta_i>0$, define a
probability on $\calp_k$ (equivalently, $S_k$) by
\begin{equation}
\pi_{\bm\eta}\lamp=\frac{Z}{z_\lam}\prod_{i=1}^k\eta_i^{a_i\lamp}\qquad\text{with}\quad
Z^{-1}=\sum_{\mu\vdash k}\frac1{z_\mu}\prod_i\eta_i^{a_i(\mu)}.
\label{220}
\end{equation}
Such multiplicative measures are classical objects of study. They are
considered in \cite{arratia} and \cite{yaku}, where many useful
cases are given. The measures $\piqt$ fall into this class with
$\eta_i=\frac{(t^i-1)}{(q^i-1)}$. If $x=(x_1,x_2,\dots)$ and
$y=(y_1,y_2,\dots)$ are two sequences of numbers and $V_\lam(X)$ is a
multiplicative basis of $\Lambda_n^k$ such as $\{e_\lam\},\
\{p_\lam\}, \{h_\lam\}$, setting $\eta_i=V_i(x)V_i(y)$ gives
$\pi_{\bm\eta}\lamp=\frac{Z}{z_\lam}V_\lam(x)V_\lam(y)$. This is in rough
analogy to the Schur measures defined in \ref{sub2d7}. For the choices
$e_\lam,\ p_\lam,\ h_\lam$, with $x_i,\ y_j$ positive numbers, the
associated measures are positive. The power sums, with all $x_i=a,\
y_i=b$, gives the Ewens measure with $\alpha=ab$. Setting $x_1=y_1=c,\ x_i=y_j=0$
otherwise, gives the measure $1/z_\lam$ after normalization. To our knowledge,
general multiplicative measures have not been previously studied.
Multiplicative systems are studied in \cite[VI Sect.~1 Ex.]{mac}.

It is natural to try out the simple rejection algorithms of
\ref{sub2d3} and \ref{sub2d4} for the measures $\pi_{\bm\eta}$. To
begin, suppose that $0<\eta_i<1$ for all $i$. The measure
$\pi_{\bm\eta}$ has the following interpretation: Pick
$\lam'\in\calp_k$ with probability $1/z_{\lam'}$. As above, for each
part of $\lam'$ of size $i$, generate a random variable taking values
1 or 0 with probability $\eta_i,1-\eta_i$. If the values for all parts
equal 1, set $\lam=\lam'$. If not, try again. For more general
$\eta_i$, divide all $\eta_i$ by $\eta_*=\max\eta_i$, and generate from
$\eta_i/\eta_*^i$. This yields the measure $\pi_{\bm\eta}$ on partitions.

Alas, this algorithm performs poorly for $\eta_i$ and $k$ in ranges of
interest. For example, with $\eta_i=\frac{t^i-1}{q^i-1}$ for
$t=2,q=4$, when $k=10,11,12,13$, the chance of success (empirically)
is $1/2000,1/4000,1/7000,1/12000$. We never succeeded in generating a
partition for any $k\geq15$.

The asymptotic distribution of the parts of a partition chosen from
$\pi_{\bm\eta}$ when $k$ is large can be studied by classical tools of
combinatorial enumeration. For fixed values of $q,t$, these problems
fall squarely into the domain of the logarithmic combinatorial
structures studied in \cite{arratia}. A series of further results for
more general $\bm\eta$ have been developed by \citeauthor{jiang}
\cite{jiang}. The following brief survey of their results gives a good
picture of typical partitions.

Of course, the theorems vary with the choice of $\eta_i$. One
convenient condition, which includes the measure $\piqt$ for fixed $q,t>1$, is
\begin{equation}
\sum_{i=1}^\infty\left|\frac{(\eta_i-1)}{i}\right|<\infty.
\label{221}
\end{equation}
\begin{thm}
Suppose $\eta_i,\ 1\leq i<\infty$, satisfy \eqref{221}. If
$\lam\in\calp_k$ is chosen from $\pi_{\bm\eta}$ of \eqref{220}, then,
for $j$ large:
\begin{gather}
\text{For any $j$, the distribution of $(a_1\lamp,\dots,a_j\lamp)$
 converges to the distribution}\notag\\
\text{of an independent Poisson vector with parameters $\eta_i/i,\ 1\leq i\leq j$.}\label{222}\\[.1in]
\text{The number of parts of $\lam$ has mean and variance asymptotic to $\log k$}\notag\\
\text{and, normalized by its mean and standard deviation,}\label{223}\\
\text{a limiting standard normal distribution.}\notag\\[.1in]
\text{The length of the $k$ largest parts of $\lam$ converge to}\notag\\
\text{the Poisson--Dirichlet distribution }\emph{\cite{gonch,bill,logan}}.\label{224}
\end{gather}
\label{thm1}
\end{thm}
These and other results from \cite{jiang,arratia} show that the parts
of a random partition are quite similar to the cycles of a unformly
chosen random permutation, with the small cycles having slightly
adjusted parameters. These results are used to give a lower bound on
the mixing time of the auxiliary variables Markov chain in Proposition
\ref{prop21} below.

\subsubsection{Simulation results}\label{sub2d6}

\begin{table}[htb]
\caption{}
\begin{small}\begin{center}\begin{tabular}{ll|ll}
Partition $\lambda\vdash10$ & Probability $\pi_{q=2,t=4}(\lambda)$&Partition $\lambda\vdash10$ & Probability $\pi_{q=2,t=4}(\lambda)$\\
\hline
10	&	0.164003&4,4,2	&	0.018177 \\
9,1	&	0.121365&4,4,1,1&	0.010098 \\
8,2	&	0.081762&4,3,3	&	0.016955	\\
8,1,1	&	0.045423&4,3,2,1&	0.030520	\\
7,3	&	0.068948&4,3,1,1,1&	0.005652	\\
7,2,1	&	0.062054&4,2,2,2&	0.004120	\\
7,1,1,1	&	0.011491&4,2,2,1,1&	0.006867	\\
6,4	&	0.063387&4,2,1,1,1,1&	0.001272	\\
6,3,1	&	0.053214&4,1,1,1,1,1,1&	0.000047	\\
6,2,2	&	0.021552&3,3,3,1&	0.004745	\\
6,2,1,1	&	0.023946&3,3,2,2&	0.005765	\\
6,1,1,1,1&	0.002217&3,3,2,1,1&	0.006405	\\
5,5	&	0.030873&3,3,1,1,1,1&	0.000593	\\
5,4,1	&	0.049942&3,2,2,2,1&	0.003459	\\
5,3,2	&	0.037734&3,2,2,1,1,1&	0.001922	\\
5,3,1,1	&	0.020963&3,2,1,1,1,1,1&	0.000214	\\
5,2,2,1	&	0.016980&3,1,1,1,1,1,1,1&	0.000006	\\
5,2,1,1,1&	0.006289&2,2,2,2,2&	0.000140	\\
5,1,1,1,1,1&	0.000349&2,2,2,2,1,1&	0.000389	\\
                       &&2,2,2,1,1,1,1&	0.000144	\\
                       &&2,2,1,1,1,1,1,1&	0.000016	\\
                       &&2,1,1,1,1,1,1,1,1&	0.000001	\\
                       &&1,1,1,1,1,1,1,1,1,1&	0.000000
\end{tabular}\end{center}\end{small}
\label{table1}
\end{table}
\begin{table}[htb]
\caption{}
\begin{small}\begin{center}\begin{tabular}{l|l|l|l}
\multicolumn{4}{c}{Sample 100-step walk for Auxiliary Variables}\\
\hline
\ \,1. 10      &26. 6,4      &51. 6,2,1,1  &\ \,76. 7,2,1\\
\ \,2. 4,3,3   &27. 10       &52. 10       &\ \,77. 10\\
\ \,3. 6,3,1   &28. 4,3,2,1  &53. 7,3      &\ \,78. 7,2,1\\
\ \,4. 5,5     &29. 8,1,1    &54. 8,2      &\ \,79. 9,1\\
\ \,5. 9,1     &30. 8,2      &55. 6,2,2    &\ \,80. 5,4,1\\
\ \,6. 8,1,1   &31. 7,3      &56. 6,4      &\ \,81. 10\\
\ \,7. 6,2,2   &32. 9,1      &57. 4,2,211  &\ \,82. 6,3,1\\
\ \,8. 9,1     &33. 8,2      &58. 5,3,2    &\ \,83. 6,3,1\\
\ \,9. 4,4,2   &34. 8,2      &59. 6,4      &\ \,84. 5,4,1\\
10. 4,4,1,1  &35. 8,2      &60. 10       &\ \,85. 8,1,1\\
11. 4,3,1,1,1&36. 10       &61. 9,1      &\ \,86. 5,3,2\\
12. 7,2,1    &37. 7,1,1,1  &62. 6,3,1    &\ \,87. 5,3,1,1\\
13. 5,3,1,1  &38. 10       &63. 4,3,3    &\ \,88. 5,2,2,1\\
14. 6,4      &39. 5,3,2    &64. 10       &\ \,89. 10\\
15. 10       &40. 4,3,3    &65. 5,5      &\ \,90. 5,3,2\\
16. 5,3,2    &41. 8,2      &66. 8,2      &\ \,91. 8,2\\
17. 4,3,3    &42. 7,3      &67. 5,4,1    &\ \,92. 5,3,2\\
18. 9,1      &43. 6,3,1    &68. 3,3,2,1,1&\ \,93. 6,3,1\\
19. 7,3      &44. 10       &69. 6,4      &\ \,94. 5,4,1\\
20. 7,3      &45. 5,5      &70. 6,1,1,1,1&\ \,95. 4,3,2,1\\
21. 5,3,2    &46. 6,3,1    &71. 4,3,2,1  &\ \,96. 7,3\\
22. 5,3,1,1  &47. 8,1,1    &72. 5,4,1    &\ \,97. 7,2,1\\
23. 5,3,1,1  &48. 6,1,1,1,1&73. 10       &\ \,98. 7,2,1\\
24. 6,3,1    &49. 10       &74. 5,2,1,1,1&\ \,99. 5,2,2,1\\
25. 5,3,2    &50. 9,1      &75. 5,2,2,1  &100. 4,2,2,1,1\\
\end{tabular}\end{center}\end{small}
\label{table2}
\end{table}

The distribution $\piqt(\lam)$ can be far from uniform. An example,
with $k=10,\ q=4,\ t=2$, is shown in \ref{table1};
$\pi_{4,2}(10)\doteq0.16,\ \pi_{4,2}(1^6)\doteq0$. The auxiliary
variables algorithm for the measure $\piqt$ has been programmed by
\citeauthor{jiang} \cite{jiang}. It seems to work well over a wide
range of $q$ and $t$. A tiny example, 100 steps when $k=10,\ q=4,\
t=2$, is shown in \ref{table2}. A comparison of the simulations with
the exact distribution (easily computed from \eqref{12} when $k=10$)
shows perfect agreement. In our experiments, the choice of $q$ and $t$
does not seriously affect the running time, and simulations seem
possible for $k$ up to $10^6$.

The distribution of the largest part, for $q=4,\ t=2$ and $k=10$,
$k=100$, and $k=1000$, based on $10^6$ steps of the algorithm, is
shown in \ref{fig1}. Comparison with the limiting results of Theorem
\ref{thm1} above seems good. The blip at the right side of the figures
comes from $(k)$; the rest of the distribution follows the limit
\eqref{224} approximately.
\begin{figure}[htb]
\centering
\includegraphics[scale=0.1, trim=0in 0in 0in 0in, clip=true]{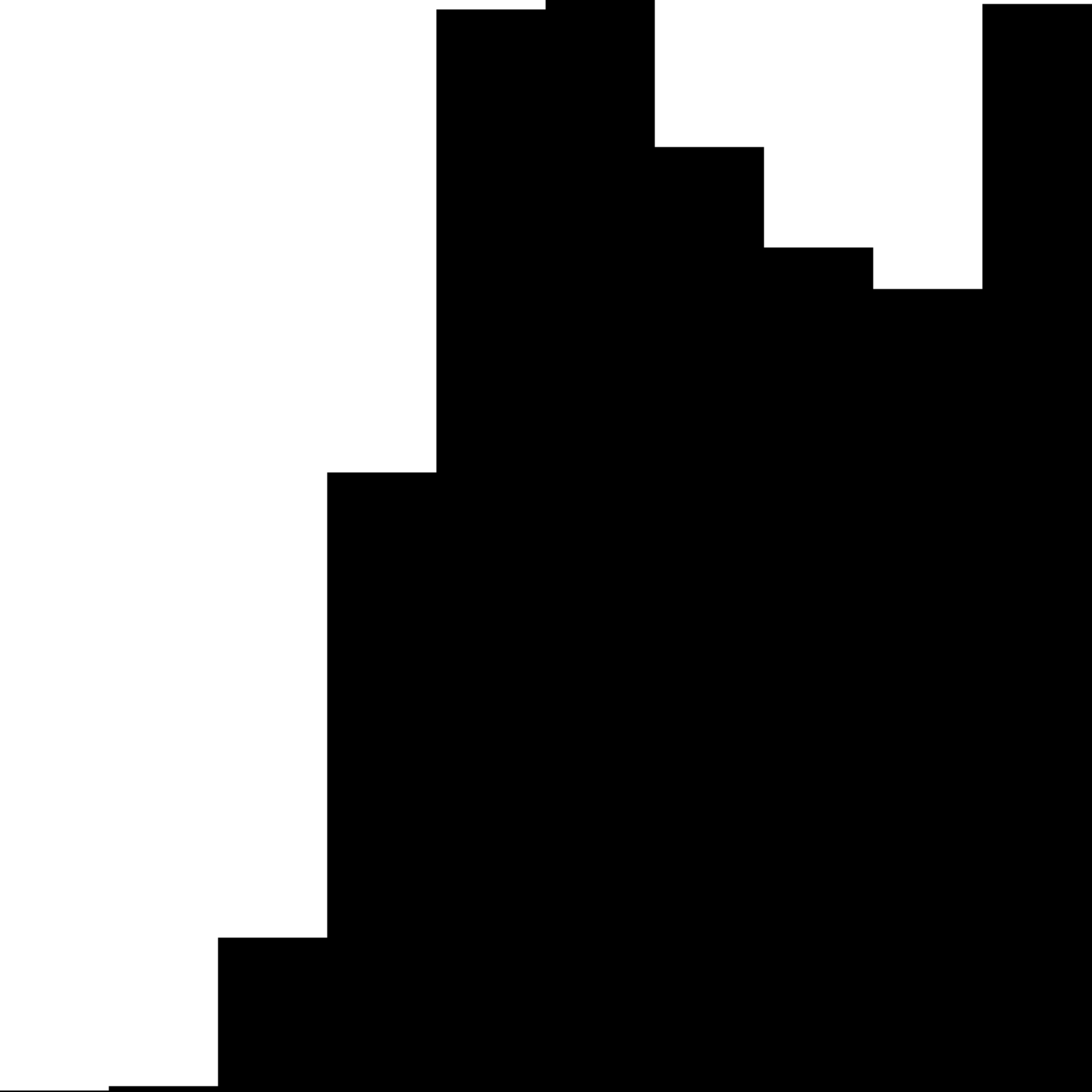}\qquad\qquad
\includegraphics[scale=0.1, trim=0in 0in 0in 0in, clip=true]{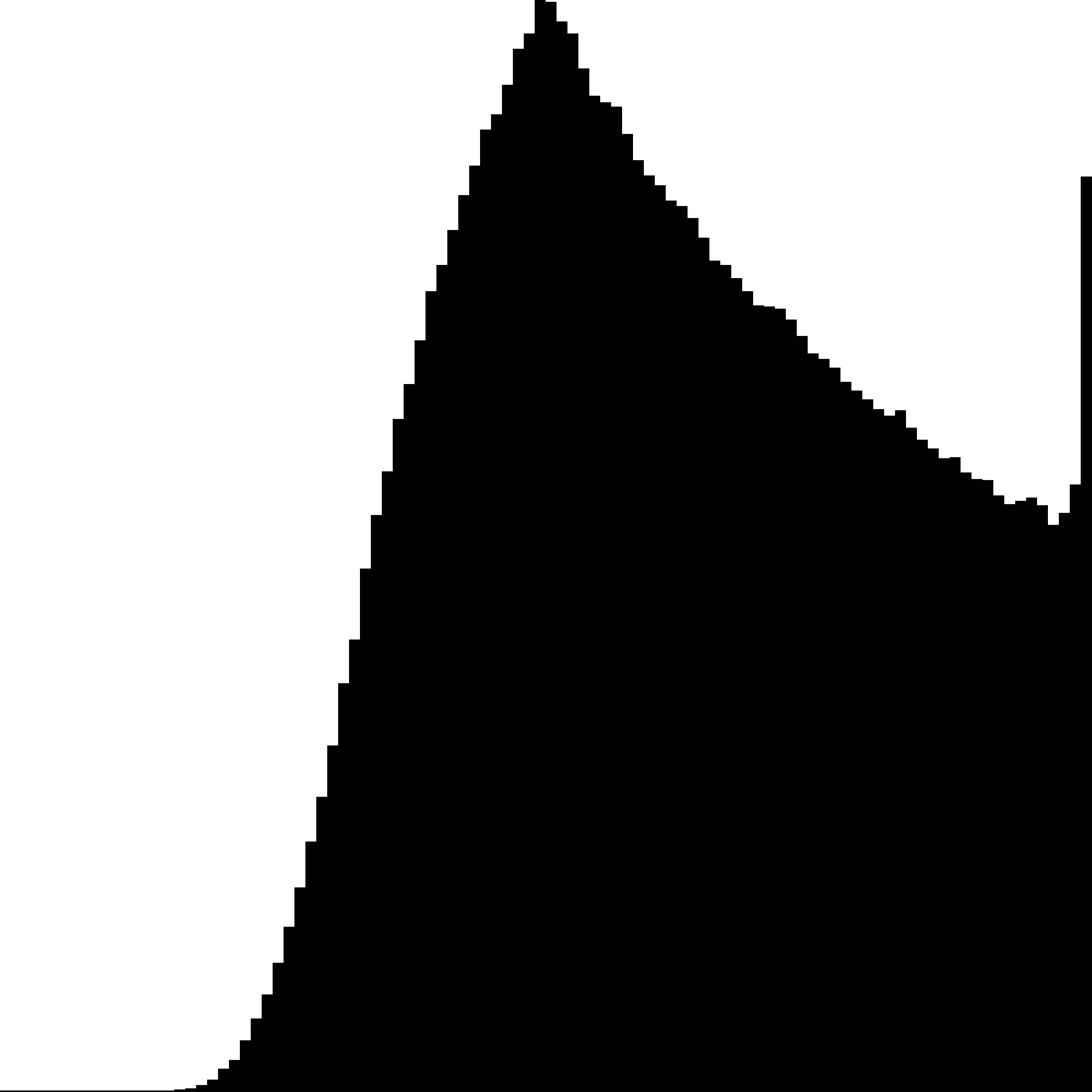}
\caption{\textit{Left}: Probability of largest part $i$ for partitions
  of $k=10$ under $\piqt$ with $q=4,\ t=2$. \textit{Right}:
  Probability of largest part $i$ for partitions of $k=100$ under
  $\piqt$ with $q=4,\ t=2$.}
\label{fig1}
\end{figure}

We have compared the auxiliary variables algorithm with the rejection
algorithm of \ref{sub2d5} and the Metropolis algorithm. As reported in
\ref{sub2d5}, rejection fails completely for $n\geq15$. The Metropolis
algorithm we used works by simulating permutations from $\piqt$ lifted
to $S_k$. From the current permutation $\sigma$, propose $\sigma'$ by
making a random transposition (all $\binom{n}{2}$ choices equally
likely). If $\piqt(\sigma')\geq\piqt(\sigma)$, move to $\sigma'$. If
$\piqt(\sigma')/\piqt(\sigma)<1$, flip a coin with probability
$\piqt(\sigma')/\piqt(\sigma)$
and move to $\sigma'$ if the coin comes up heads; else stay at
$\sigma$. For small values of $k$, Metropolis is competitive with
auxiliary variables. \citeauthor{jiang} have computed the mixing time
for $k=10,20,30,40,50$ by a clever sampling algorithm. For $q=4,t=2$,
the following table shows the number of steps required to have total
variation distance less than $1/10$ starting from the partition $(k)$.  Also 
shown is $p(k)$, the number of partitions of $k$, to give a feeling for the 
size of the state space. 
\begin{table}[htb]
\begin{center}\begin{tabular}{r|rrrrr}
$k$&10&20&30&40&50\\\hline
Aux&1&1&1&1&1\\
Met&8&17&26&37&53\\
$p(k)$&42&627&5604&37338&204,226
\end{tabular}\end{center}
\end{table}

The theorems of \ref{sec3} show that auxiliary variables requires a
bounded number of steps for arbitrary $k$. In the computations above,
the distance to stationarity after one step of the auxiliary variables
is 0.093 (within a 1\% error in the last decimal) for $k=10,\dots,50$.
For larger $k$ (e.g., $k=100$), the Metropolis algorithm seemed to
need a very large number of steps to move at all. This is consistent
with other instances of auxiliary variables, such as the Swendsen--Wang
algorithm for the Ising and Potts model (away from the critical
temperature; see \cite{borgs}).

\subsubsection{Other measures on partitions}\label{sub2d7}

This portmanteau section gives pointers to some of the many other
measures that have been studied on $\calp_k,S_k$. Often these studies
are fascinating, deep, and extensive. All measures studied here seem
distinct from $\piqt$.

A remarkable two-parameter family of measures on partitions has been
introduced by Jim Pitman. For $\theta\geq0,\ 0\leq\alpha\leq1$, and
$\lam\vdash k$ with $\ell$ parts, set
\begin{equation*}
P_{\theta,\alpha}\lamp=\frac{k!}{z_\lam}\frac{\theta^{(\alpha,\ell-1)}}{(\theta+1-\alpha)^{(1,k-1)}}\prod_{j=1}^k\left[(1-\alpha)^{(1,j-1)}\right]^{a_j\lamp},
\end{equation*}
where
\begin{equation*}
\theta^{(a,m)}=\begin{cases}1 &\text{if }m=0,\\
\theta(\theta+a)\dots(\theta+(m-1)a)&\text{for }m=1,2,3,\dots.\end{cases}
\end{equation*}
These measures specialize to $1/z_\lam\ (\theta=\ 1,\alpha=\ 0)$, and
the Ewens measure ($\theta$ fixed, $\alpha=\ 0)$, see \cite[sec.\
3.2]{pitman}. They arise in a host of probability problems connected
to stable stochastic problems of index $\alpha$. They are also being
used in applied probability connected to genetics and Bayesian
statistics. They satisfy elegant consistency properties as $k$ varies.
For example, deleting a random part gives the corresponding measure on
$\calp_{k-1}$. For these and many other developments, see the
book-length treatments of \cite{bert}, \cite[sec.\ 3.2]{pitman}.

One widely studied measure on partitions is the Plancherel measure,
\begin{equation*}
p\lamp=f\lamp^2/k!,
\end{equation*}
with $f\lamp$ the dimension of the irreducible representation of $S_k$
associated to shape $\lam$. This measure was perhaps first studied in
connection with Ulam's problem on the distribution of the length of
the longest increasing sequence in a random permutation; see
\cite{logan,versik}. For extensive developments and references, see
\cite{kerov,pd49}.

The Schur measures of \cite{ok-inf,ok-symm,ok-uses,borodin} are
generalizations of the Plancherel measure. Here the chance of $\lam$
is taken as proportional to $s_\lam(\bm{x})s_\lam(\bm{y})$, with
$s_\lam$ the Schur function and $\bm{x},\bm{y}$ collections of
real-valued entries. Specializing $\bm{x}$ and $\bm{y}$ in various
ways yields a variety of previously studied measures. One key
property, if the partition is ``tilted 135$^\circ$'' to make a v-shape
and the local maxima projected onto the $x$-axis, the resulting points
form a determinantal point process with a tractable kernel. This gives
a fascinating collection of shape theorems for the original partition.

One final distribution, the uniform distribution on $\calp_k$, has
also been extensively studied. For example, a uniformly chosen
partition has order $\pi/\sqrt{6k}$ parts of size 1, the largest part
is of size $(\sqrt{6k}/\pi)\cdot \log(\sqrt{6k}/\pi)$, the number of
parts is of size $\sqrt{6k}\log(k/(2\pi))$.  A survey with much more
refined results is in \cite{fris}.

The above only scratches the surface. The reader is encouraged to look
at \cite{ok-inf,ok-symm,ok-uses} to see the breadth and depth of the
subject as applied to Gromov--Witten theory, algebraic geometry, and
physics. The measures there seem closely connected to the ``Plancherel
dual'' of our $\piqt$. This dual puts mass proportional to $c\lamp
c'\lamp$ on $\lam$, with $c,c'$ the arm-leg length products defined in
\ref{sec3a} below.

\section{Main results}\label{sec3}

This section shows that the auxiliary variables Markov chain $M$ with
stationary distribution $\piqt\lamp,\ \lam\in\calp_k$, is explicitly
diagonalizable with eigenfunctions $f_\lam(\mu)$ essentially the
coefficients of the Macdonald polynomials expanded in the power sum
basis. The result is stated in \ref{sec3a}. The proof, given in
\ref{sec3c}, is somewhat computational. An explanatory overview is in
\ref{sec3b}. In \ref{sec5}, these eigenvalue/eigenvector results are
used to bound rates of convergence of $M$.

\subsection{Statement of main results}\label{sec3a}

Fix $q,t>1$ and $k\geq2$. Let
$M(\lam,\mu)=\sum_{\laml}w_\lam(\laml)M_{\laml}(\lam,\mu)$ be the
auxiliary variables Markov chain on $\calp_k$. Here, $w_\lam(\cdot)$
and $M_{\laml}(\lam,\mu)$ are defined in \eqref{212}, \eqref{213}, and
studied in \ref{sub2d3} and \ref{sub2d4}.
For a partition $\lambda$, let
\begin{equation}
c_\lam(q,t)=\prod_{s\in\lam}\left(1-q^{a(s)}t^{l(s)+1}\right)
\quad\text{and}\quad 
c_\lam'(q,t)=\prod_{s\in\lam}\left(1-q^{a(s)+1}t^{l(s)}\right),
\label{32}
\end{equation}
where the product is over the boxes in the shape $\lam$, and $a(s)$ is the arm
length and $l(s)$ the leg length of box $s$ \cite[VI (8.1)]{mac}.%

\begin{thm}
\item[(1)] The Markov chain $M(\lam,\nu)$ is reversible and ergodic with
 stationary distribution $\piqt\lamp$ defined in \eqref{12}. This
 distribution is properly normalized.
\item[(2)] The eigenvalues of $M$ are $\{\beta_\lam\}_{\lam\in\calp_k}$
 given by
\begin{equation*}
\beta_\lam=\frac{t}{q^k-1}\sum_{i=1}^{\ell\lamp}\left(q^{\lam_i}-1\right)t^{-i}.
\end{equation*}
Thus,
$\beta_k=1,\beta_{k-1,1}=\frac{t}{q^k-1}\left(\frac{q^{k-1}-1}{t}+\frac{q-1}{t^2}\right),\dots$.
\item[(3)] The corresponding right eigenfunctions are
\begin{equation*}
f_\lam(\rho)=\xrl(q,t)\prod_{i=1}^{\ell(\rho)}(1-q^{\rho_i})
\end{equation*}
with $\xrl(q,t)$ the coefficients occurring in the following
expansion of the Macdonald polynomials in terms of the power sums
\emph{\cite[VI (8.19)]{mac}}:
\begin{equation}
P_\lam(x;q,t)
=\frac1{c_\lam(q,t)}\sum_\rho\left[z_\rho^{-1}\prod_{i=1}^{\ell(\rho)}(1-t^{\rho_i})\xrl(q,t)\right]p_\rho(x),
\label{31}
\end{equation}

\item[(4)] The $f_\lam(\rho)$ are orthogonal in $L^2(\piqt)$ with
\begin{equation*}
\<f_\lam,f_\mu\>=\delta_{\lam\mu}c_\lam(q,t)c_\lam'(q,t)\frac{(q,q)_k}{(t,q)_k}.
\end{equation*}
\label{thm2}
\end{thm}
\begin{example}
When $k=2$, from \eqref{215}, the matrix $M$ with rows and columns
indexed by 2, $1^2$, is
\begin{equation*}
M=\frac1{2t}\begin{pmatrix}t+1&t-1\\[.1in]
\dfrac{(q-1)(t+1)}{q+1}&\dfrac{4t+(q-1)(t-1)}{q+1}\end{pmatrix}.
\end{equation*}
Macdonald \cite[p.~359]{mac} gives tables of $K(\lam,\mu)$ for $2\leq
k\leq6$. For $k=2$, $K(\lam,\mu)$ is
$\left(\begin{smallmatrix}1&q\\t&1\end{smallmatrix}\right)$. The
character matrix is
$\left(\begin{smallmatrix}1&1\\-1&1\end{smallmatrix}\right)$, and the
product is
$\left(\begin{smallmatrix}1-q&1+q\\t-1&t+1\end{smallmatrix}\right)$.
From Theorem \ref{thm2}(3), the rows of this matrix, multiplied
coordinate-wise by $(1-q^2),(1-q)^2$, give the right eigenvectors:
\begin{align*}
f_{(2)}(2)&=f_{(2)}(1^2)=(1-q)^2(1+q),\\ 
f_{(1^2)}(2)&=(t-1)(1-q^2),\\
f_{(1^2)}(1^2)&=(t+1)(1-q)^2.
\end{align*}
Then $f_{(2)}(\rho)$ is a constant function, and $f_{(1^2)}(\rho)$ satisfies
$\sum_\rho M(\lam,\rho)f_{(1^2)}(\rho)=\beta_{(1^2)}f_{(1^2)}(\lam)$, with
$\beta_{(1^2)}=\frac{1+t^{-1}}{1+q}$.
\end{example}

Further useful formulae, used in \ref{sec5}, are \cite[VI Sect.~8
Ex.~8]{mac}:
\begin{equation}\begin{aligned}
X_\rho^{(k)}(q,t)&=(q,q)_k\prod_{i=1}^{\ell(\rho)}(1-t^{\rho_i}) \qquad
&X_\rho^{(1^k)}(q,t)&=(-1)^{|\rho|-\ell(\rho)}(t,t)_k\prod_{i=1}^{\ell(\rho)}(1-t^{\rho_i})^{-1}\\
X_{(k)}^\lam(q,t)&=\prod_{\substack{(i,j)\in \lam\\(i,j)\neq(1,1)}}\left(t^{i-1}-q^{j-1}\right) \qquad 
&X_{(1^k)}^\lam(q,t)&=\frac{c_\lam'(q,t)}{(1-t)^k}\sum_T\varphi_T(q,t)
\end{aligned}\label{33}
\end{equation}
with the sum over standard tableaux $T$ of shape $\lam$, and
$\varphi_T(q,t)$ from \cite[VI p.~341 (1)]{mac} and \cite[VI
(7.11)]{mac}.

\subsection{Overview of the argument}\label{sec3b}

Macdonald \cite[VI]{mac} defines the Macdonald polynomials as the
eigenfunctions of the operator $D_{q,t}^1:\Lambda_n\to\Lambda_n$ from \eqref{22}. 
%Here,
%for $f(x_1,\dots,x_n)\in\Lambda_n$, let
%$T_{q,x_i}f(x_1,\dots,x_n)=f(x_1,\dots,qx_i,\dots,x_n)$. Then \cite[VI
%(3.4)]{mac},
%\begin{equation*}
%D_{q,t}^1=\sum_{i=1}^nA_i(x,t)T_{q,x_i},\qquad A_i(x,t)=\prod_{j\neq i}\frac{tx_i-x_j}{x_i-x_j}.
%\end{equation*}
As described in \ref{sec2a} above, $D_{q,t}^1$ is self-adjoint for the
Macdonald inner product and sends $\Lambda_n^k$ into itself \cite[VI
(4.15)]{mac}. For $\lam\vdash k,\ k\leq n$,
\begin{equation}
D_{q,t}^1\plam=\barb_\lam\plam,\quad\text{with}\quad
\barb_\lam=\sum_{i=1}^{\ell\lamp}q^{\lam_i}t^{k-i}.
\label{34}
\end{equation}
The Markov chain $M$ is related to an affine rescaling of the operator
$D_{q,t}^1$, which \cite[VI (4.1)]{mac} calls $E_n$. We work directly with
$D_{q,t}^1$ to give direct access to Macdonald's formulae. The affine
rescaling is carried out at the end of \ref{sec3c} below.

The \textit{integral form} of Macdonald polynomials \cite[VI Sect.~8]{mac}
is
\begin{equation*}
J_\lam(x;q,t)=c_\lam(q,t)\plam
\end{equation*}
for $c_\lam$ defined in \eqref{32}. Of course, the $J_\lam$ are also
eigenfunctions of $D_{q,t}^1$. The $J_\lam$ may be expressed in terms of
the shifted power sums via \cite[VI (8.19)]{mac}:
\begin{equation}
J_\lam(x;q,t)=\sum_\rho z_\rho^{-1}\xrl(q,t)p_\rho(x;t),\qquad p_\rho(x;t)
=p_\rho(x)\prod_{i=1}^{\ell(\rho)}(1-t^{\rho_i}).
\label{35}
\end{equation}
This is our equation \eqref{31} above. 
%Extending \cite[VI Ex.~?]{mac},
%see also \cite[App.~B Prop.~2]{awata}. 
In Proposition \ref{prop21} below, we compute the action of
$D_{q,t}^1$ on the power sum basis: for $\lam$ with $\ell$ parts,
\begin{equation}\begin{aligned}
D_{q,t}^1 p_\lam
&\defeq\sum_\mu\barm(\lam,\mu)p_\mu\\
&=[n]p_\lam+\frac{t^n}{t-1}\sum_{J\subseteq\{1,\dots,\ell\}}p_{\lam_{J^c}}
\prod_{k\in J}\left(q^{\lam_k}-1\right)
\sum_{\mu\vdash|\lam_J|}\prod_m\left(1-t^{-\mu_m}\right)\frac{p_\mu}{z_\mu}.
\end{aligned}\label{36}
\end{equation}
On the right, the coefficient of $p_{\lam_{J^c}}p_\mu$ is
essentially the Markov chain $M$; we use $\barm$ for this
unnormalized version. Indeed, we first computed \eqref{36} and then
recognized the operator as a special case of the auxiliary variables
operator.

Equations \eqref{34}--\eqref{36} show that simply scaled versions of
$\xrl$ are eigenvectors of the matrix $\barm$ as follows. From
\eqref{34}, \eqref{35},
\begin{equation}\begin{aligned}
\barb_\lam P_\lam(x;q,t)&=D_{q,t}^1\plam=\frac1{c_\lam}D_{q,t}^1(J_\lam)\\
&=\frac1{c_\lam}D_{q,t}^1\left(\sum_\rho\xrl\frac1{z_\rho}p_\rho(x;t)\right)
=\frac1{c_\lam}\sum_\rho\xrl\frac{\prod(1-t^{\rho_i})}{z_\rho}D_{q,t}^1 p_\rho(x)\\
&=\frac1{c_\lam}\sum_\rho\frac{\prod(1-t^{\rho_i})}{z_\rho}\xrl\sum_\mu\barm(\rho,\mu)p_\mu(x)\\
&=\frac1{c_\lam}\sum_\mu p_\mu\sum_\rho\frac{\prod(1-t^{\rho_i})}{z_\rho}\xrl\barm(\rho,\mu).
\end{aligned}\label{37}
\end{equation}
Also, from \eqref{34} and \eqref{35},
\begin{equation}
\barb_\lam\plam=\frac{\barb_\lam}{c_\lam}J_\lam(x;q,t)=\frac{\barb_\lam}{c_\lam}\sum_\mu
X_\mu^\lam\frac1{z_\mu}\prod_i(1-t^{\mu_i})p_\mu(x).
\label{38}
\end{equation}
Equating coefficients of $p_\mu(x)$ on both sides of \eqref{37},
\eqref{38}, gives
\begin{equation}
\frac{\barb_\lam}{c_\lam}X_\mu^\lam\frac1{z_\mu}\prod_i(1-t^{\mu_i})=\frac1{c_\lam}\sum_\rho
\frac{\xrl}{z_\rho}\prod_i(1-t^{\rho_i})\barm(\rho,\mu).
\label{39}
\end{equation}
This shows that
$h_\lam(\mu)=\frac{X_\mu^\lam\prod_i(1-t^{\mu_i})}{z_\mu}$ is a left
eigenfunction for $\barm$ with eigenvalue $\barb_\lam$. It follows
from reversibility
$(\piqt(\rho)\barm(\rho,\mu)=\piqt(\mu)\barm(\mu,\rho))$ that
$h_\lam(\mu)/\piqt(\mu)$ is a right eigenfunction for $\barm$. Since
$\piqt(\mu)=Zz_\mu^{-1}(q,t)$, simple manipulations give the formulae
of part (3) of Theorem \ref{thm2}.

As explained in \ref{sec2a} above, the Macdonald polynomials
diagonalize a family of operators $D_{q,t}^r,\ 0\leq r\leq n$.  The
argument above applies to all of these. In essence, the method
consists of interpreting equations such as \eqref{35} as linear
combinations of partitions, equating $p_\lam$ with $\lam$.

\subsection{Proof of Theorem \ref{thm2}}\label{sec3c}

As in \ref{sec2a} above, let $\dqt(z)=\sum_{r=0}^n\dqt^rz^r$. Let
$[n]=\sum_{i=1}^nt^{n-i}$. The main result identifies $\dqt^1$,
operating on the power sums, as an affine transformation of the
auxiliary variables Markov chain. The following Proposition is the
first step, providing the expansion of $\dqt^1$ acting on power sums. 
A related computation is in
\cite[App.~B Prop.~2]{awata}.
\begin{prop}
\item[(a)]  If $f$ is homogeneous, then
\begin{equation*}
D_{q,t}^0f = f, \quad D_{q,t}^n f = q^{\dg(f)} f, \quad\text{and}\quad
D_{q,t}^{n-1} f = t^{\dg(f) + n(n-1)/2}q^{\dg(f)}D_{q^{-1},t^{-1}}^1 f.
\end{equation*}
\item[(b)] If $\lam=(\lam_1,\dots,\lam_\ell)$ is a partition then
\begin{multline}
D_{q,t}^1 p_\lam=[n]p_\lam+\\\sum_{J\subseteq\{ 1,\dots,\ell\}\atop J\neq\emptyset}
p_{\lam_{J^c}}\left(\prod_{k\in J}\left(q^{\lam_k}-1\right)\right)\frac{t^n}{t-1} 
\sum_{\mu\vdash |\lam_J|}\left(\prod_{m=1}^{\ell(\mu)} \left(1-t^{-\mu_m}\right)\right)\frac1{z_\mu}p_\mu.
\label{310}
\end{multline}
\label{prop21}
\end{prop}
\begin{proof}[Proof of Proposition \ref{prop21}]
\item[(a)] If $f$ is homogeneous then
\begin{equation}
D_{q,t}^n f = \sum_{I\subseteq \{1,\dots,n\} \atop |I|=n}
A_I(x;t)\prod_{i\in I }T_{q,x_i} f= T_{q,x_1}T_{q,x_2}\dots T_{q,x_n} 
f=q^{\dg(f)} f.
\label{311}
\end{equation}
By definition,
\begin{equation}
A_I(x;t) = \frac1{a_\delta}\left(\prod_{i\in I} T_{t,x_i}\right) a_\delta
=t^{r(r-1)/2}\prod_{i\in I\atop j\not\in I} \frac{tx_i-x_j}{x_i-x_j}.
\label{312}
\end{equation}
Letting $x^\gamma =x_1^{\gamma_1}\cdots x_n^{\gamma_n}$ for $\gamma=(\gamma_1,\ldots, \gamma_n)$,
\begin{align*}
T_{q,x_1}T_{q,x_2}\dots\hat{T}_{q,x_j}\dots T_{q,x_n} x^\gamma
&= q^{\gamma_1+\dots+\gamma_n-\gamma_j}x_1^{\gamma_1}\dots x_n^{\gamma_n} \\
&=q^{\dg(x^\gamma)}q^{-\gamma_j}x^\gamma=q^{\dg(x^\gamma)}T_{q^{-1}, x_j}x^\gamma,
\end{align*}
and it follows that
\begin{equation}
T_{q,x_1}T_{q,x_2}\dots\hat{T}_{q,x_j}\dots T_{q,x_n} f=q^{\dg(f)} T_{q^{-1},x_j} f,
\label{313}
\end{equation}
if $f$ is homogeneous. Thus,
\begin{align*}
D_{q,t}^{n-1} f&=\sum_{I\subseteq \{1,\dots,n\} \atop |I|=n-1}A_I(x;t)\left(\prod_{i\in I }T_{q,x_i}\right) f\\
&=\sum_{j=1}^n A_{\{j\}^c}(x;t) T_{q,x_1}\dots \hat{T}_{q,x_j} \dots T_{q,x_n} f \\
&=\sum_{j=1}^n\frac1{a_\delta} T_{t,x_1}\dots \hat{T}_{t,x_j} \dots T_{t,x_n} a_\delta T_{q,x_1}\dots \hat{T}_{q,x_j} \dots T_{q,x_n} f \\
&=\sum_{j=1}^n\frac1{a_\delta} t^{\dg(f)+\dg(a_\delta)} T_{t^{-1},x_j} a_\delta q^{\dg(f)} T_{q^{-1},x_j} f \\
&=t^{\dg(f) + n(n-1)/2} q^{\dg(f)} \sum_{j=1}^n A_j\left(x;t^{-1}\right)T_{q^{-1},x_j} f \\
&=t^{\dg(f) + n(n-1)/2} q^{\dg(f)} D_{q^{-1},t^{-1}}^1 f.
\end{align*}
Hence,
\begin{equation}
D_{q,t}^{n-1} f=t^{\dg(f) + n(n-1)/2} q^{\dg(f)}D_{q^{-1},t^{-1}}^1 f.
\label{314}
\end{equation}
\item[(b)] By \cite[VI (3.7),(3.8)]{mac}
\begin{align*}
D_{1,t}(z)m_\lam&=\sum_{\beta\in S_n\lam}\left(\prod_{i=1}^n\left(1+zt^{n-i}\right)\right) s_\beta
=\left(\prod_{i=1}^n\left(1+zt^{n-i}\right)\right)\sum_{\beta\in S_n\lam} s_\beta \\
&=\left(\prod_{i=1}^n\left(1+zt^{n-i}\right)\right) m_\lam
=\sum_{r=0}^nt^{r(r-1)/2}\begin{bmatrix}n\\r\end{bmatrix}z^rm_\lam,
\end{align*}
where $m_\lam$ denotes the monomial symmetric function.
Thus, since $D_{q,t}(z)=\sum_{r=0}^nD_{q,t}^rz^r$ and 
\begin{equation*}
D_{1,t}^r = \sum_{I\subseteq \{1,\dots,n\} \atop |I|=r}
A_I(x;t)\prod_{i\in I }T_{1,x_i}= \sum_{I\subseteq \{1,\dots,n\} \atop |I|=r}A_I(x;t),
\end{equation*}
it follows that
\begin{equation}\label{D1t1f}%%%%%%(3.15)
\sum_{j=1}^n A_j(x;t) f = D_{1,t}^1 f =[n]f
\end{equation}
for a symmetric function $f$. By \cite[VI Sect.~3 Ex.~2]{mac},
\begin{equation}\label{Dqt1xr}%%%%%%(3.16)
(t-1)\sum_{i=1}^nA_i(x;t)x_i^r=t^ng_r\left(x;0,t^{-1}\right)-\delta_{0r}, 
\end{equation}
where, from \cite[VI (2.9)]{mac},
\begin{equation*}
g_r(x;q,t) = \sum_{\lam\vdash n} z_\lam(q,t)^{-1}p_\lam(x) , 
\end{equation*}
with $z_\lam(q,t)$ as in \eqref{11}.

Let $\lam = (\lam_1,\dots,\lam_\ell)$ be a partition, and 
\begin{equation*}
\text{for }J\subseteq \{1,\dots,\ell\},\qquad\text{let }\lam_J=(\lam_{j_1},\dots,\lam_{j_k})
\quad\text{if }J=\{j_1,\dots,j_k\}.
\end{equation*}
Then, using
\begin{equation}
T_{q,x_i} p_r = q^r x_i^r - x_i^r + p_r = (q^r-1)x_i^r + p_r,
\label{317}
\end{equation}
\eqref{Dqt1xr}, and \eqref{D1t1f},
\begin{align*}
D_{q,t}^1 p_\lam&= \sum_{j=1}^n A_j(x;t) T_{q,x_j} p_{\lam_1}\dots p_{\lam_\ell} \\
&=\sum_{j=1}^nA_j(x;t)\left(\left(q^{\lam_1}-1\right)x_j^{\lam_1}+ 
 p_{\lam_1}\right)\dots\left(\left(q^{\lam_\ell}-1\right)x_j^{\lam_\ell}+p_{\lam_\ell}\right)\\
&= \sum_{j=1}^n A_j(x;t) \sum_{J\subseteq\{ 1,\dots, \ell\}}\left(\prod_{k\in J}\left(q^{\lam_k}-1\right)\right)x_j^{|\lam_J|}\prod_{s\not\in J} p_{\lam_s}\\
&= \sum_{J\subseteq\{ 1,\dots, \ell\}}\prod_{s\not\in J} p_{\lam_s}\left(\prod_{k\in J}\left(q^{\lam_k}-1\right)\right)\sum_{j=1}^n A_j(x;t)x_j^{|\lam_J|}\\
&= \sum_{j=1}A_j(x;t)p_\lam+\sum_{J\subseteq\{ 1,\dots, \ell\}\atop J\ne \emptyset}p_{\lam_{J^c}}\left(\prod_{k\in J}\left(q^{\lam_k}-1\right)\right)
  \frac{t^n}{t-1} g_{|\lam_J|}\left(x;0,t^{-1}\right)\\
&= [n]p_\lam+\sum_{J\subseteq\{ 1,\dots, \ell\}\atop J\ne \emptyset}p_{\lam_{J^c}}\left(\prod_{k\in J}\left(q^{\lam_k}-1\right)\right)\frac{t^n}{t-1} 
  \sum_{\mu\vdash |\lam_J|} \frac1{z_\mu\left(0;t^{-1}\right)} p_\mu\\
&= [n]p_\lam+\sum_{J\subseteq\{ 1,\dots, \ell\}\atop J\ne \emptyset}p_{\lam_{J^c}}\left(\prod_{k\in J}\left(q^{\lam_k}-1\right)\right)\frac{t^n}{t-1} 
  \sum_{\mu\vdash |\lam_J|}\left(\prod_{m=1}^{\ell(\mu)} \left(1-t^{-\mu_m}\right)\right)\frac1{z_\mu} p_\mu.\qedhere
\end{align*}
\end{proof}

Let us show that the measure $\piqt\lamp$ is properly
normalized and compute the normalization of the eigenvectors.

\begin{lemma}
  Let $\piqt(\lam)$ be as in \eqref{12} and let
  $f_\lam(\rho)=\xrl\prod_{i=1}^{\ell(\rho)}(1-q^{\rho_i})$ be as in
  Theorem \ref{thm2}(3) Then
\begin{equation*}
\sum_{\lam\vdash k}\piqt\lamp=1\quad\text{and}\quad
\sum_{\rho\vdash k}f_\lambda^2(\rho)\pi_{q,t}(\lambda)=\frac{(q,q)_k}{(t,q)_k}c_\lam c_\lam'.
\end{equation*}
\label{cor1}
\end{lemma}
\begin{proof}[Proof of Lemma \ref{cor1}]
  From \cite[VI (2.9),(4.9)]{mac}, the Macdonald polynomial
  $P_{(k)}(x;q,t)$ can be written
\begin{equation*}
P_{(k)}=\frac{(q,q)_k}{(t,q)_k}\cdot
g_k=\frac{(q,q)_k}{(t,q)_k}\sum_{\lam\vdash k}z_\lam(q,t)^{-1}p_\lam.
\end{equation*}
From \cite[VI (4.11),(6.19)]{mac},
$$\< P_\lam, P_\lam\> = c_\lam'/c_\lam,$$
and it follows that
\begin{align*}
\sum_{\lam\vdash k}\piqt\lamp
&=\frac{(q,q)_k}{(t,q)_k}\sum_{\lam\vdash k}z_\lam(q,t)^{-1} 
=\frac{(t,q)_k}{(q,q)_k}\left\<P_{(k)},P_{(k)}\right\>
=\frac{(t,q)_k}{(q,q)_k}\frac{c_{(k)}'}{c_{(k)}}=1.
\end{align*}
To get the normalization of
$f_\lam(\rho)=\xrl\prod_{i=1}^{\ell(\rho)}(1-q^{\rho_i})$ in Theorem
\ref{thm2}(4), use \eqref{35} and
\begin{align*}
c_\lam c_\lam'
&= (c_\lam)^2\<P_\lam,P_\lam\>
=\<J_\lam,J_\lam\> \\
&=\sum_{\rho\vdash k}z_\rho^{-2}\left(\xrl(q,t)\prod_{i=1}^{\ell(\rho)}(1-t^{\rho_i})\right)^2\<p_\rho,p_\rho\>\\
&=\sum_{\rho\vdash k}z_\rho^{-1}\left(\xrl(q,t)\prod(1-t^{\rho_i})\right)^2\prod_{i=1}^{\ell(\rho)}\frac{(1-q^{\rho_i})}{(1-t^{\rho_i})}\\
&=\sum_{\rho\vdash k}f_\lam^2(\rho)z_\rho^{-1}(q,t)
=\frac{(t,q)_k}{(q,q)_k}\sum_{\rho\vdash k}f_\lam^2(\rho)\piqt\lamp.\qedhere
\end{align*}
\end{proof}

We next show that an affine renormalization of the discrete version
$\barm$ \eqref{35} of the Macdonald operator equals the auxiliary
variables Markov chain of \ref{sec2c}. Along with Macdonald \cite[VI
(4.1)]{mac}, define
\begin{equation*}
E_k=t^{-k}\dqt^1-\sum_{i=1}^kt^{-i},\qquad\text{and let }\tilde{E}_k=\frac{t}{q^k-1}E_k, 
\end{equation*}
operating on $\Lambda_n^k$. From
\eqref{33}, the eigenvalues of $E_k$ are
$\beta_\lam=\sum_{i=1}^{\ell\lamp}(q^{\lam_i}-1)t^{-i}$. Noting that
$\beta_{(k)}=\frac{q^k-1}{t}$, the operator
$\tilde{E}_k$ is a normalization of $E_k$ with top
eigenvalue $1$. From Proposition \ref{prop21}(b), for $\lam$ a partition
with $\ell$ parts,
\begin{equation*}
\tilde{E}_kp_\lam=\frac1{(1-t^{-1})(q^k-1)}
\sum_{\substack{J\subseteq\{1,\dots,\ell\}\\J\neq\emptyset}}\prod_{k\in J}
\left(q^{\lam_k}-1\right)p_{\lam_{J^c}}
\sum_{\mu\vdash|\lam_J|}\prod_{i=1}^{\ell(\mu)}\left(1-t^{-k}\right)\frac{p_\mu}{z_\mu}. 
\end{equation*}
Using $p_\lam$ as a surrogate for $\lam$ as in \ref{sec3b}, the
coefficient of $\nu=\lam_{J^c}\mu$ is exactly $M(\lam,\nu)$ of
\ref{sec2c}.

This completes the proof of Theorem \ref{thm2}.\qed
\begin{example}
When $k=2$, from the definitions
\begin{align*}
\tilde{E}_2p_2&=\left(\frac{1-t^{-1}}2\right)p_{1^2}+\left(\frac{1+t^{-1}}2\right)p_{2},\\
\tilde{E}_2p_{1^2}&=\frac1{2(q+1)}\left((q+3-qt^{-1}+t^{-1})p_{1^2}
+(q-1)(1+t^{-1})p_{2}\right).
\end{align*}
Thus, on partitions of 2, the matrix of $\tilde{E}_2$ is
\begin{equation*}
\begin{pmatrix}\dfrac{1+t^{-1}}2&\dfrac{1-t^{-1}}2\\[.15in]
\dfrac{(1+t^{-1})(q-1)}{2(q+1)}&\dfrac{3+q+t^{-1}-t^{-1}q}{2(q+1)}\end{pmatrix}.
\end{equation*}
This is the matrix of \eqref{215} derived there from the probabilistic
description.
\end{example}

\section{Jack polynomials and Hanlon's walk}\label{sec4}

The Jack polynomials are a one-parameter family of bases for the
symmetric polynomials, orthogonal for the weight
$\<p_\lam,p_\mu\>_\alpha=\alpha^{\ell(\lam)}z_\lam\delta_{\lam\mu}$.
They are an important precursor to the full two-parameter Macdonald
polynomial theory, containing several classical bases: the limits
$\alpha=0,\ \alpha=\infty$, suitably interpreted, give the
$\{e_\lam\},\{m_\lam\}$ bases; $\alpha=1$ gives Schur functions;
$\alpha=2$ gives zonal polynomials for $GL_n/O_n$; $\alpha=\frac12$
gives zonal polynomials for $GL_n(\mathbb{H})/U_n(\mathbb{H})$ where
$\mathbb{H}$ is the quaternions (see \cite[VII]{mac}).  A good deal of
the combinatorial theory for Macdonald polynomials was first developed
in the Jack case. Further, the Jack theory has been developed in more
detail \cite{han,stanley,knop} and \cite[VI Sect.~10]{mac}.

Hanlon \cite{han} managed to interpret the differential operators
defining the Jack polynomials as the transition matrix of a Markov
chain on partitions with stationary distribution
$\pi_\alpha\lamp=Z\alpha^{-\ell\lamp}/z_\lam$, described in
\ref{sub2d2} above. In later work \cite{pd86}, this Markov chain was
recognized as the Metropolis algorithm for generating $\pi_\alpha$
from the proposal of random transpositions. This gives one of the few
cases where this important algorithm can be fully diagonalized. See
\cite{obata} for a different perspective.

Our original aim was to extend Hanlon's findings, adding a second
``sufficient statistic'' to $\ell\lamp$, and discovering a
Metropolis-type Markov chain with the Macdonald coefficients as
eigenfunctions. It did not work out this way. The auxiliary variables
Markov chain makes more vigorous moves than transpositions, and there
is no Metropolis step.  Nevertheless, as shown below, Hanlon's chain
follows from interpreting a limiting case of $D_\alpha^1$, one of
Macdonald's $D_\alpha^r$ operators. We believe that all of the
operators $D_\alpha^r$ should have interesting interpretations. In
this section, we derive Hanlon's chain from the Macdonald operator
perspective.

\subsubsection*{Overview}

There are several closely related operators used to develop the Jack
theory. Macdonald \cite[VI Sect.~3 Ex.~3]{mac} uses $\dal(u)$ and
$\dal^r$, defined by
\begin{equation}
\dal(u)=\sum_{r=0}^n\dal^ru^{n-r}
=\frac1{a_\delta}\sum_{w\in S_n}\det(w)x^{w\delta}\prod_{i=1}^n\left(u+(w\delta)_i
+\alpha x_i\frac{\partial}{\partial x_i}\right)
\label{41}
\end{equation}
where $\delta=(n-1,n-2,\dots,1,0)$, $a_\delta$ is the Vandermonde
determinant, and $x^\gamma = x_1^{\gamma_1}\cdots x_n^{\gamma_n}$ for
$\gamma=(\gamma_1,\dots,\gamma_n)$. He shows \cite[VI Sect.~3
Ex.~3c]{mac} that
\begin{equation}
\dal(u)=\lim_{t\to1}\frac{z^n}{(t-1)^n}D_{t^\alpha,t}\left(z^{-1}\right)
\qquad\text{if }z=(t-1)u-1,
\label{42}
\end{equation}
so that the Jack operators are a limiting case of Macdonald
polynomials.

Macdonald \cite[VI Sect.~4 Ex.~2b]{mac} shows that the Jack
polynomials $J_\lam^\alpha$ are eigenfunctions of $\dal(u)$ with
eigenvalues $\beta_\lam(\alpha)=\prod_{i=1}^n(u+n-i+\alpha\lam_i)$.
Stanley \cite[Pf.\ of Th.~3.1]{stanley} and Hanlon \cite[(3.5)]{han}
use $D(\alpha)$ defined as follows. Let
\begin{gather}
\partial_i=\frac{\partial}{\partial x_i},\qquad
U_n=\frac12\sum_{i=1}^nx_i^2\partial_i^2,\qquad
V_n=\sum_{i\neq j}\frac{x_i^2}{x_i-x_j}\partial_i,\label{43}\\
\text{and}\quad D(\alpha)=\alpha U_n+V_n.\label{44}
\end{gather}

Hanlon computes the action of $D(\alpha)$ on the power sums in the form (see \eqref{47})
\begin{equation}
D(\alpha)p_\lam=(n-1)rp_\lam+\alpha\binom{r}{2}\sum_\mu \ell_{\mu\lam}(\alpha)p_\mu,
\label{45}
\end{equation}
where $n$ is the number of variables and $\lam$ is a partition of $r$.

The matrix $\ell_{\mu\lam}(\alpha)$ can be interpreted as the transition
matrix of the following Markov chain on the symmetric group $S_r$. For
$w\in S_r$, set $c(w)=$ \# cycles. If the chain is currently at $w_1$,
pick a transposition $(i,j)$ uniformly; set $w_2=w_1(i,j)$. If
$c(w_2)=c(w_1)+1$, move to $w_2$. If $c(w_2)=c(w_1)-1$, move to $w_2$
with probability $1/\alpha$; else stay at $w_1$. This Markov chain has
transition matrix
\begin{equation*}
H_\alpha(w_1, w_2)=\begin{cases}
\dfrac1{\binom{r}2}&\text{if }w_2=w_1(i,j)\text{ and }c(w_2)=c(w_1)+1\\[.175in]
\dfrac1{\alpha\binom{r}2}&\text{if }w_2=w_1(i,j)\text{ and }c(w_2)=c(w_1)-1\\[.175in]
\dfrac{n(w_1)(1-\alpha^{-1})}{\binom{r}2}&\text{if }w_1=w_2\end{cases}
\end{equation*}
where $n(w_1)=\sum_i(i-1)\lam_i$ for $w_1$ of cycle type $\lam$.
Hanlon notes that this chain only depends on the conjugacy class of
$w_1$, and the induced process on conjugacy classes is still a Markov
chain for which the transition matrix is the matrix of
$\ell_{\mu\lam}(\alpha)$ of \eqref{45}.

The Jack polynomial theory now gives the eigenvalues of the Markov
chain $H_\alpha(w_1,w_2)$, and shows that the corresponding
eigenvectors are the coefficients when the Jack polynomials are
expanded in the power sum basis. The formulae available for Jack
polynomials then allow for a careful analysis of rates of convergence
to stationarity; see \cite{pd86}.

We may see this from the present perspective as follows.
\begin{prop}
  Let $\dal(u)$ and $D(\alpha)$ be defined by \eqref{41}, \eqref{44}.
\item[(a)] Let $\dal^t$ be the coefficient of $u^{n-t}$ in $\dal(u)$
 (see \emph{\cite[VI Sect.~3 Ex.~3d]{mac}}). If $f$ is a homogeneous
 polynomial in $x_1,\dots,x_n$ of degree $r$, then
\begin{equation}
\dal^0f=f,\quad\dal^1(f)=\left(\alpha r+\frac12n(n-1)\right)f,
\quad\text{and}\quad\dal^2f=(-\alpha^2U_n-\alpha V_n+c_n)f,
\label{46}
\end{equation}
where
\begin{equation*}
c_n=\frac12\alpha^2r(r-1)+\frac12\alpha rn(n-1)+\frac1{24}n(n-1)(n-2)(3n-1).
\end{equation*}
\item[(b)] From \emph{\cite[Pf.\ of Th.~3.1]{stanley}},
\begin{equation}
D(\alpha)p_\lam=\frac12 p_\lam\left(\begin{array}{l}
\displaystyle\sum_{k=1}^s\alpha\lam_k(\lam_k-1)+\alpha\sum_{j,k=1\atop j\ne k}^s\frac{\lam_j\lam_kp_{\lam_j+\lam_k}}{p_{\lam_j}p_{\lam_k}}\\
\qquad+\displaystyle\sum_{k=1}^s\lam_k(2n-\lam_k-1)+
\sum_{k=1}^s\frac{\lam_k}{p_{\lam_k}}\sum_{m=1}^{\lam_k-1} p_{\lam_k-m}p_m\end{array}\right).
\label{47}
\end{equation}
\label{prop41}
\end{prop}
\begin{remark}
 From part (a), up to affine rescaling, $\dal^2$ is the
 Stanley--Hanlon operator. From part (b), this operates on the power
 sums in precisely the way that the Metropolis algorithm operates.
 Indeed, multiplying a permutation $w$ by a transposition $(i,j)$
 changes the number of cycles by one; the change takes place by
 fusing two cycles (the first term in \eqref{47}) or by breaking one
 of the cycles in $w$ into parts (the second term in \eqref{47}). The
 final term constitutes the ``holding'' probability from the
 Metropolis algorithm.
\end{remark}
\begin{proof}[Proof of Proposition \ref{prop41}]
\item[(a)] $\dal^0$ is the cofficient of $u^n$ in $\dal(u)$, so
\begin{equation*}
\dal^0f = \frac1{a_\delta} \sum_{w \in S_n}\det(w) x^{w\delta} f= \frac1{a_\delta} a_\delta f = f.
\end{equation*}
$\dal^1$ is the coefficient of $u^{n-1}$ in $\dal(u)$, so
\begin{align*}
\dal^1f&=\frac1{a_\delta} \sum_{w \in S_n}\det(w)x^{w\delta}
\sum_{i=1}^n\left((w\delta)_i+\alpha x_i\partial_i\right)f\\
      &=\frac1{a_\delta}\sum_{w\in S_n}\det(w)x^{w\delta}\left(\alpha r+\sum_{i=1}^n(n-i)\right)f
       =\frac{a_\delta}{a_\delta}\left(\alpha r+\binom{n}{2}\right)f.
\end{align*}
$\dal^2$ is the coefficient of $u^{n-2}$ in $\dal(u)$, so
\begin{equation}\label{Dalpha2}%%%%%%(4.8)
\dal^2 f = \frac1{a_\delta} \sum_{w\in S_n} \left( \det(w)x^{w\delta} \prod_{1 \leq i<j \leq n} 
\left((w\delta)_i + \alpha x_i\partial_i\right) 
\left((w\delta)_j + \alpha x_j\partial_j\right) \right)f.
\end{equation}
Since $x_i\partial_i x_j\partial_j=x_j\partial_j x_i\partial_i$ for
all $1\leq i,\ j\leq n$, and $x_i\partial_i
x_j\partial_j=x_ix_j\partial_i\partial_j$ for $i\neq j$,
\begin{equation*}
r^2=\left(\sum_{i=1}^n x_i \partial_i \right)^2= \sum_{i=1}^n x_i  \partial_i x_i\partial_i
+ 2\sum_{1\leq i<j\leq n}^n x_ix_j\partial_i\partial_j
\end{equation*}
so that the coefficient of $\alpha^2$ in \eqref{Dalpha2} is
\begin{align*}
\frac1{a_\delta}\sum_{w\in S_n}\det(w)x^{w\delta}\left(\sum_{1\leq i<j\leq n}x_i\partial_i x_j\partial_j\right)f
&=\sum_{1\leq i<j\leq n}x_ix_j\partial_i\partial_j f\\ 
&=\frac12\left(r^2-\sum_{i=1}^nx_i\partial_i x_i\partial_i \right)f\\
&=\frac12\left(r^2-\sum_{i=1}^nx_i\partial_i+x_i^2\partial_i^2\right)f\\
&=\left(\frac12r^2-\frac12 r-\frac12\sum_{i=1}^nx_i^2\partial_i^2\right)f\\
&=\left(\frac12(r^2-r)-U_n\right)f.
\end{align*}
The coefficient of $\alpha^0$ in equation \eqref{Dalpha2} is
\begin{align*}
\frac1{a_\delta}\sum_{w\in S_n}&\left(\det(w)x^{w\delta}\sum_{1\leq i<j\leq n}(w\delta)_i(w\delta)_j\right)f\\
&=\frac1{a_\delta}\sum_{w\in S_n}\left(\det(w)x^{w\delta}\frac12\left(\left(\sum_{i=1}^n(w\delta)_i\right)^2
    -\sum_{i=1}^n (w\delta_i)^2 \right)\right)f\\
&= \frac12 \left(\frac14n^2(n-1)^2 -\frac16n(n-1)(2n-1)\right)f\\ 
&= \frac1{24}n(n-1)(n-2)(3n-1)f
\end{align*}
since, for each $w\in S_n$,
\begin{align*}
\sum_{i=1}^n(w\delta)_i&=\sum_{i=1}^ni-1=\frac12 n(n-1)\quad\text{and}\quad
\sum_{i=1}^n(w\delta)_i^2&=\sum_{i=1}^n(i-1)^2=\frac16n(n-1)(2n-1).
\end{align*}
Since $a_\delta=\sum_{w\in S_n}\det(w)x^{w\delta}=\prod_{1\leq i<j\leq n}(x_i-x_j)$,
then, for fixed $i$,
\begin{align*}
\frac1{a_\delta}\sum_{w\in S_n}\det(w)(w\delta)_ix^{w\delta}&=\frac1{a_\delta}\sum_{w\in S_n}\det(w)x_i\partial_ix^{w\delta}\\
&=\frac1{a_\delta}x_i\partial_i\sum_{w\in S_n}\det(w)x^{w\delta}=\frac1{a_\delta}x_i\partial_ia_\delta\\
&=\frac1{a_\delta}x_i\left(\sum_{j=1}^{i-1}\frac{a_\delta}{x_j-x_i}\partial_i(x_j-x_i)
    +\sum_{j=i+1}^n\frac{a_\delta}{x_i-x_j}\partial_i(x_i-x_j)\right)\\
&=\frac1{a_\delta}x_i\left(a_\delta\sum_{j\neq i}\frac1{x_i-x_j}\right)\\
&=\sum_{j\neq i}\frac{x_i}{x_i-x_j}
\end{align*}
so that the coefficient of $\alpha$ in \eqref{Dalpha2} is
\begin{align*}
\frac1{a_\delta}&\sum_{w\in S_n} \det(w) x^{w\delta}\sum_{1 \leq i < j \leq n}\left( (w\delta)_i x_j\partial_j
  + (w\delta)_j x_i\partial_i\right) f\\
&=\frac1{a_\delta}\sum_{w\in S_n} \det(w) x^{w\delta}\left(\sum_{i=1}^nx_i\partial_i \sum_{j=1}^n (w\delta)_j
  - \sum_{i=1}^n (w\delta)_i x_i \partial_i \right)f \\
&=\frac12 n(n-1)rf-\frac1{a_\delta} \sum_{w\in S_n} \det(w) x^{w\delta}\sum_{i=1}^n (w\delta)_i x_i\partial_i f \\
&=\frac12 n(n-1)rf-\sum_{i=1}^n \left(\frac1{a_\delta} \sum_{w\in S_n} \det(w) x^{w\delta}(w\delta)_i\right) x_i\partial_i f \\
&=\frac12 n(n-1)rf-\sum_{i=1}^n\sum_{j\neq i} \frac{x_i}{x_i-x_j}x_i\partial_i f\\ 
&=\left(\frac12 rn(n-1)-V_n\right) f.
\end{align*}
\item[(b)] Since $\partial_i p_\ell = \ell x_i^{\ell-1}$,
\begin{align*}
\alpha U_n p_\lam &=
\frac{\alpha}2\sum_{i=1}^n x_i^2\partial_i^2 p_\lam
=\frac{\alpha}2\sum_{i=1}^n x_i^2\partial_i^2 p_{\lam_1}p_{\lam_2}\dots p_{\lam_s}\\
&=\frac{\alpha}2\sum_{i=1}^n x_i^2\partial_i
\left(\sum_{j=1}^s\frac{\lam_jx_i^{\lam_j-1}}{p_{\lam_j}}p_\lam\right)\\
&=\frac{\alpha}2\sum_{i=1}^n x_i^2
\left(\sum_{j,k=1\atop j\ne k}^s
\frac{\lam_j\lam_k x_i^{\lam_j-1}x_i^{\lam_k-1}}{p_{\lam_j}p_{\lam_k}}p_\lam
  +\sum_{j=1}^s\frac{\lam_j(\lam_j-1)}{p_{\lam_j}} x_i^{\lam_j-2}p_\lam\right)\\
&=\frac{\alpha}2\left(\sum_{j,k=1\atop j\ne k}^s
\sum_{i=1}^n\frac{\lam_j\lam_k x_i^{\lam_j+\lam_k}}{p_{\lam_j}p_{\lam_k}}p_\lam
  +\sum_{j=1}^s \frac{\lam_j(\lam_j-1)}{p_{\lam_j}} x_i^{\lam_j} p_\lam \right) \\
&=\frac{\alpha}2\left(\sum_{j=1}^s\lam_j(\lam_j-1)p_\lam+\sum_{j,k=1\atop j\ne k}^s\frac{\lam_j\lam_kp_{\lam_j+\lam_k}}{p_{\lam_j}p_{\lam_k}}p_\lam\right).
\end{align*}
Since
\begin{equation*}
p_{\lam_k-m}p_m=\left( \sum_{i=1}^n x_i^{\lam_k-m}\right)\left( \sum_{j=1}^n x_j^m\right)
= \sum_{i=1}^n x_i^{\lam_k} + \sum_{i,j=1\atop i\ne j}^n x_i^{\lam_k-m}x_j^m,
\end{equation*}
then
\begin{equation*}
p_{\lam_k-m}p_m - p_{\lam_k} = \sum_{i,j=1\atop i\ne j}^n x_i^{\lam_k-m}x_j^m.
\end{equation*}
Hence
\begin{align*}
V_np_\lam &= \sum_{i,j=1\atop i\ne j}^n \frac{x_i^2}{x_i-x_j}\partial_i p_\lam
= \sum_{i,j=1\atop i\ne j}^n\frac{x_i^2}{x_i-x_j}
\sum_{k=1}^s \frac{\lam_k x_i^{\lam_k-1}}{p_{\lam_k}} p_\lam \\
&= \sum_{1\leq i<j\leq n} \frac{x_i^{\lam_k+1}-x_j^{\lam_k+1}}{x_i-x_j}
\sum_{k=1}^s \frac{\lam_k}{p_{\lam_k}} p_\lam \\
&= \sum_{k=1}^s \frac{\lam_k}{p_{\lam_k}}p_\lam 
\left(\sum_{1\leq i<j\leq n} x_i^{\lam_k} + x_i^{\lam_k-1}x_j
   + \dots + x_i x_j^{\lam_k-1} + x_j^{\lam_k}\right) \\
&= \sum_{k=1}^s \frac{\lam_k}{p_{\lam_k}}p_\lam 
\frac12\left(\sum_{i,j = 1\atop i\ne j}^n x_i^{\lam_k} + x_i^{\lam_k-1}x_j
   + \dots + x_i x_j^{\lam_k-1} + x_j^{\lam_k}\right) \\
&= \sum_{k=1}^s \frac{\lam_k}{p_{\lam_k}}p_\lam 
\frac12\left((n-1)p_{\lam_k} 
+\sum_{m=1}^{\lam_k-1}\left(p_{\lam_k-m}p_m-p_{\lam_k}\right)
   +(n-1)p_{\lam_k} \right) \\
&= \sum_{k=1}^s \frac{\lam_k}{p_{\lam_k}}p_\lam 
\frac12\left(\left(2(n-1)-(\lam_k-1)\right)p_{\lam_k}
   +\sum_{m=1}^{\lam_k-1} p_{\lam_k-m}p_m \right) \\
&= \frac12\left(\sum_{k=1}^s\lam_k(2n-\lam_k-1)p_{\lam}+\sum_{k=1}^s\frac{\lam_k}{p_{\lam_k}}p_\lam\sum_{m=1}^{\lam_k-1}p_{\lam_k-m}p_m\right) \\
&= \frac12p_\lam\left(\sum_{k=1}^s\lam_k(2n-\lam_k-1)+\sum_{k=1}^s\frac{\lam_k}{p_{\lam_k}}\sum_{m=1}^{\lam_k-1}p_{\lam_k-m}p_m\right).
\end{align*}
The formula for $D(\alpha)p_\lam = (\alpha U_n + V_n)p_\lam$ now follows.
\end{proof}

\section{Rates of convergence}\label{sec5}

This section uses the eigenvectors and eigenvalues derived above to
give rates of convergence for the auxiliary variables Markov chain.
\ref{sec51} states the main results: starting from the partition
$(k)$ a bounded number of steps suffice for convergence, independent
of $k$. \ref{sec52} contains an overview of the argument and needed
lemmas. \ref{sec53} gives the proof of Theorem \ref{thm51}, and
\ref{sec54} develops the analysis starting from $(1^k)$, showing that
$\log_q(k)$ steps are needed.

\subsection{Statement of main results}\label{sec51}

Fix $q,t>1$ and $k\geq2$. Let $\calp_k$ be the partitions of $k,\
\piqt\lamp=Z/z_\lam(q,t)$ the stationary distribution defined in
\eqref{12}, and $M(\lam,\nu)$ the auxiliary variables Markov chain
defined in Proposition \ref{prop1}. The total variation distance
$\|M_{(k)}^\ell-\piqt\|_{\text{TV}}$ used below is defined in \eqref{27}.
\begin{thm}  Consider the auxiliary variables Markov chain on 
partitions of $k\ge 4$.  
Then, for all $\ell\ge 2$
\begin{equation}
4\left\|M_{(k)}^\ell-\piqt\right\|^2_{\text{TV}}\leq 
 \frac{1}{(1-q^{-1})^{3/2}(1-q^{-2})^2}
\left(\frac{1}{q}+\frac{1}{tq^{k/2}}\right)^{2\ell} + k\left(\frac{t}{t-1}\right)
\left(\frac{2}{q^{k/4}}\right)^{2\ell}.
\label{51}
\end{equation}
\label{thm51}
\end{thm}

For example, if $q=4$, $t=2$, and $k=10$ the bound becomes
$1.76(.26)^{2\ell}+20(1/512)^{2\ell}$.  Thus, when $\ell=2$ the total
variation distance is at most .05 in this example.

\subsection{Outline of proof and basic lemmas}\label{sec52}

Let $\{f_\lam,\beta_\lam\}_{\lam\vdash k}$ be the eigenfunctions and
eigenvalues of $M$ given in Theorem \ref{thm2}. From \ref{sec2b}, for
any starting state $\rho$,
\begin{equation}
4\|M_\rho^\ell-\piqt\|_{\text{TV}}^2
\leq\sum_\lam\frac{\left(M^\ell(\rho,\lam)-\piqt\lamp\right)^2}{\piqt\lamp}
=
\sum_{\lam\neq(k)}\barf_\lam^2(\rho)\beta_\lam^{2\ell}
\label{52}
\end{equation}
with $\barf_\lam$ right eigenfunctions normalized to have norm one.
At the end of this subsection we prove the following:
\begin{gather}
\sum_\lam\barf_\lam^2(\rho)
=\frac{1}{\piqt(\rho)},\qquad\text{for any }\rho\in\calp_k.\label{53}\\
\left(\frac{1-t^{-k})}{1-t^{-1}}\right)\frac{1}{k\piqt(k)}
\quad\text{is an increasing sequence bounded by}\quad (1-q^{-1})^{-1/2}.
\label{54}\\
\beta_\lam\text{ is monotone increasing in the usual partial order (moving up boxes);}\label{55}\\
\text{in particular, $\beta_{k-1,1}$ is the second largest eigenvalue and all }\beta_\lam>0.\notag\\
\beta_{k-r,r}\sim\frac{2}{q^r}.\label{56}
\end{gather}

Using these results, consider the sum on the right side of \eqref{52},
for $\lam$ with largest part $\lam_1$ less than $k-r$. Using
monotonicity, \eqref{55}, and the bound \eqref{53},
\begin{equation}
\sum_{\lam:\lam_1\leq k-r}\barf_\lam^2(k)\beta_\lam^{2\ell}
\leq\left(\frac{2}{q^r}\right)^\ell\piqt^{-1}(k)\leq \frac{t}{t-1}\left(\frac{2}{q^r}\right)^\ell k,
\label{57}
\end{equation} 
By taking $r=k/4$ gives the second term on the right hand side of
\eqref{51}.

Using monotonicity again,
\begin{equation}
\sum_{\lam\ne(k)\atop \lam_1>k-j^*}\barf_\lam^2(k)\beta_\lam^{2\ell}
\leq\sum_{r=1}^{j^*}\beta_{(k-r,r)}^{2\ell}\sum_{\gamma\vdash r}\barf_{(k-r,\gamma)}^2(k).
\label{58}
\end{equation} 
The argument proceeds by looking carefully at $\barf_\lam^2$ and
showing
\begin{equation}
\barf_{(k-r,\gamma)}^2(k)\leq c\barf_{\gamma}^2(r)
\label{59}
\end{equation}
for a constant $c$. In \eqref{59} and throughout this section, $c=c(q,t)$ denotes a
positive constant which depends only on $q$ and $t$, but not on $k$. Its value
may change from line to line.
Using \eqref{53} on $\calp_r$ shows
$\sum_{\lam'\vdash r}\barf_{\lam'}^2(r)=\piqt^{-1}(r)\sim cr$. Using this
  and \eqref{56} in \eqref{58} gives an upper bound
\begin{equation}
\sum_{\lam\neq(k)\atop \lambda_1\ge k-j^*}\barf_\lam^2(\rho)\beta_\lam^{2\ell}
\leq c\sum_{r=1}^{j^*}\left(\frac{2}{q^r}\right)^\ell r.
\label{510}
\end{equation} 
This completes the outline for starting state $(k)$. 

This section concludes by proving the preliminary results announced
above.
\begin{lemma}
  For any $\rho\in\calp_k$, the normalized eigenfunctions
  $\barf_\lam(\rho)$ satisfy
\begin{equation*}
\sum_{\lam\vdash k}\barf_\lam(\rho)^2=\frac{1}{\piqt(\rho)}.
\end{equation*}
\label{lem52}
\end{lemma}
\begin{proof}
  The $\{\barf_\lam\}$ are orthonormal in $L^2(\piqt)$. Fix $\rho\in \calp_k$ and 
  let $\delta_\rho(\nu) = \delta_{\rho\nu}$ be the measure concentrated at $\rho$.
  Expand the function $g(\nu)=\delta_\rho(\nu)/\piqt(\rho)$ in this basis:
  $g(\nu)=\sum_\lam\<g|\barf_\lam\>\barf_\lam(\nu)$. Using the
  Plancherel identity, $\sum
  g(\nu)^2\piqt(\nu)=\sum_\lam\<g|\barf_\lam\>^2$. Here, the left side
  equals $\piqt^{-1}(\rho)$ and $\<g|\barf_\lam\>=\sum_\nu
  g(\nu)\barf_\lam(\nu)\piqt(\nu)=\barf_\lam(\rho)$. So the right side
  is the needed sum of squares.
\end{proof}

The asymptotics in \eqref{54} follow from the following lemma.
\begin{lemma}
  For $q,t>1$, the sequence
\begin{equation*}
P_k=\left(\frac{1-t^{-k}}{1-t^{-1}}\right)\frac1{k\pi_{q,t}(k)}
=\frac{(1-q^{-k})}{(1-t^{-1})}\frac{q^k}{t^k}\frac{(t,q)_k}{(q,q)_k}
=\prod_{j=1}^{k-1}\frac{1-t^{-1}q^{-j}}{1-q^{-j}}
\end{equation*}
is increasing and bounded by $\displaystyle\frac1{\sqrt{1-q^{-1}}}$.
\label{lem51}
\end{lemma}
\begin{proof}
  The equalities follow from the definitions of $\pi_{q,t}(\lam)$,
  $(t,q)_k$ and $(q,q)_k$. Since
  $\displaystyle\frac{1-t^{-1}q^{-k}}{1-q^{-k}}$ $>1$,
  the sequence is increasing. The bound follows from
\begin{align*}
\prod_{j=1}^\infty \frac{1-t^{-1}q^{-j}}{1-q^{-j}}
&= \exp\left(\sum_{j=1}^\infty \log(1-t^{-1}q^{-j})-\log(1-q^{-j})\right) \\
&= \exp\left(\sum_{j=1}^\infty \sum_{n=1}^\infty 
\big( \frac{q^{-jn}}{n} - \frac{t^{-n}q^{-jn}}{n}\big)\right) \\
&= \exp\left(\sum_{n=1}^\infty \sum_{j=1}^\infty 
\frac{q^{-jn}(1-t^{-n})}{n} \right) \\
&= \exp\left(\sum_{n=1}^\infty 
\frac{(1-t^{-n})}{n}\frac{q^{-n}}{1-q^{-n}} \right) \\
&= \exp\left(\sum_{n=1}^\infty 
\frac{(1-t^{-n})}{q^n-1}\frac{1}{n} \right) \\
&\leq \exp\left( \sum_{n=1}^\infty \frac{1}{2q^nn} \right) \\
&= \exp\left(-\frac12\log(1-q^{-1})\right) 
= \frac{1}{\sqrt{1-q^{-1}}}.\qedhere
\end{align*}
\end{proof}
\begin{remark}
  The function $P_\infty = \lim_{k\to \infty} P_k$ is an analytic function of $q,t$ for
  $|q|,|t|>1$, thoroughly studied in the classical theory of
  partitions \cite[Sect.~2.2]{andrews}. 
\end{remark}

For the next lemma, recall the usual dominance partial order on
$\calp_k:\lam\geq\mu$ if $\lam_1+\dots+\lam_i\geq\mu_1+\dots+\mu_i$
for all $i$ \cite[I.1]{mac}. This amounts to ``moving up boxes'' in
the diagram for $\mu$. Thus $(k)$ is largest, $(1^k)$ smallest. When
$k=6$, $(5,1)>(4,2)>(3,3)$, but (3,3) and (4,1,1) are not comparable.
The following result shows that the eigenvalues $\beta_\lam$ are
monotone in this order. A similar monotonicity holds for the random
transpositions chain \cite{pd150}, the Ewens sampling chain \cite{pd86},
and the Hecke algebra deformation chain \cite{pd45}.
\begin{lemma}
  For $q,t>1$, the eigenvalues
\begin{equation*}
  \beta_\lam=\frac{t}{q^k-1}\sum_{j=1}^{\ell\lamp}\left(q^{\lam_j}-1\right)t^{-j}
\end{equation*}
  are monotone in $\lam$.
\label{lem53}
\end{lemma}
\begin{proof}
  Consider first a partition $\lam,\ i<j$, with
  $a=\lam_i\geq\lam_j=b$, where moving one box from row $i$ to row $j$
  is allowed.  It must be shown that
  $q^{a+1}t^{-i}+q^{b-1}t^{-j}>q^at^{-i}+q^bt^{-j}$. Equivalently,
\begin{align*}
q^{a+1}+q^{b-1}t^{-(j-i)}&>q^a+q^bt^{-(j-i)}\\
\text{or}\qquad q^{a+1}t^{j-i}+q^{b-1}&>q^at^{j-i}+q^b\\
\text{or}\qquad q^at^{j-i}(q-1)&>q^{b-1}(q-1).
\end{align*}
Since $t^{j-i}>1$ and $q^{a-b+1}>1$,
this always holds.
\end{proof}

By elementary manipulations,
$\displaystyle\frac{q^a-1}{q^b-1}<\frac{q^a}{q^b}=\frac1{q^{b-a}}$ for
$1<a<b$, so that
\begin{equation}
\beta_{(k-r,r)}=\frac{t}{q^k-1}\left(\frac{q^{k-r}-1}{t}+\frac{q^r-1}{t^2}\right)
\leq\frac1{q^r}+\frac1{tq^{k-r}}=\frac1{q^r}\left(1+\frac1{tq^{k-2r}}\right),
\label{Arunsbeta}
\end{equation}
which establishes \eqref{56}.

\subsection{Proof of Theorem \ref{thm51}}\label{sec53}

From Theorem \ref{thm2}, the normalized eigenvectors are given by
\begin{equation}
\barf_\lam(k)^2=\frac{\left(X_{(k)}^\lam\left(q^k-1\right)\right)^2}{c_\lam c_\lam'}\cdot
\frac{(t,q)_k}{(q,q)_k},\quad\text{where}\quad
X_{(k)}^\lam=\prod_{\substack{(i,j)\in\lam\\(i,j)\neq(1,1)}}\left(t^{i-1}-q^{j-1}\right)
\label{512}
\end{equation}
and $c_\lam$ and $c_\lam'$ are given by \eqref{32}.

\begin{lemma}
  For $\lambda=(k-r,\gamma)$, with $\gamma\vdash r$ and $r\leq k/2$, 
\begin{equation*}
\barf_\lam(k)^2\leq \barf_{\gamma}(r)^2
\frac{\left(1-q^{-k}\right)^2}{\left(1-q^{-r}\right)^2} \frac{q^k}{t^k}\frac{(t,q)_k}{(q,q)_k}
\frac{t^r}{q^r}\frac{(q,q)_r}{(t,q)_r}.
\end{equation*}
\label{cor51}
\end{lemma}

\begin{proof}

\begin{figure}[h]
\centering
\includegraphics[scale=0.5, trim=0in 0in 0in 0in, clip=true]{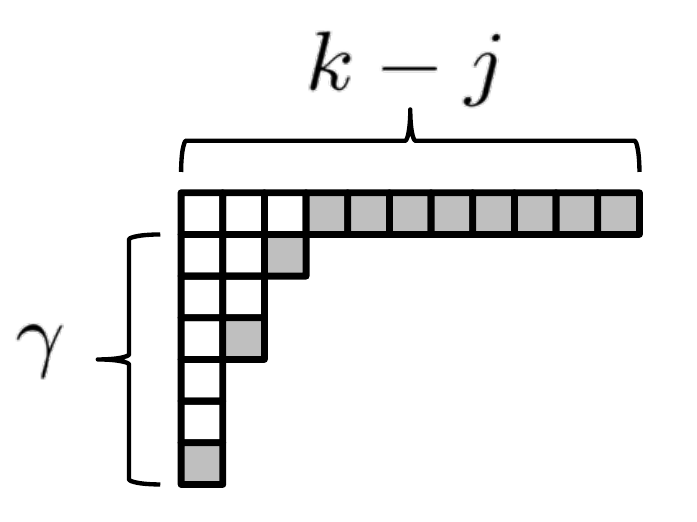}
%\caption{}
\label{diagram}
\end{figure}

  Let $\lambda=(k-r,\gamma)$ with $\gamma\vdash r$ and $r\leq k/2$.
  Let $U$ be the boxes in the first row of $\lam$, and let $L$ be the
  shaded boxes in the figure above.

  For a box $s$ in $\lambda$, let $i(s)$ be the row number and $j(s)$
  the column number of $s$. Then
\begin{align*}
\left(\frac{X_{(k)}^\lam}{X_{(r)}^\gamma}\right)^2
&=\frac{\prod_{\substack{(i,j)\in\lam\\(i,j)\neq(1,1)}}\left(t^{i-1}-q^{j-1}\right)^2 }{\prod_{\substack{(i,j)\in\gamma\\(i,j)\neq(1,1)}}\left(t^{i-1}-q^{j-1}\right)^2}\\
&=\prod_{s\in L}\left(t^{i(s)-1}-q^{j(s)-1}\right)^2 =\prod_{s\in U}\left(t^{l(s)}-q^{j(s)-1}\right)^2\\
&=\prod_{m=1}^{\gamma_1}\left(t^{\gamma'_m}-q^{m-1}\right)^2\prod_{m=\gamma_1+1}^{k-r}\left(1-q^{m-1}\right)^2\\
\end{align*}
where $\gamma_m'$ is the length of the $m$th column of $\gamma$.
Next,
\begin{align*}
\frac{c_\lam c_\lam'}{c_\gamma c_\gamma'}&=\frac{\prod_{s\in\lam}\left(1-q^{a(s)}t^{l(s)+1}\right)\left(1-q^{a(s)+1}t^{l(s)}\right)}
                                         {\prod_{s\in\gamma}\left(1-q^{a(s)}t^{l(s)+1}\right)\left(1-q^{a(s)+1}t^{l(s)}\right)}\\ 
&=\prod_{s\in U}\left(1-q^{a(s)}t^{l(s)+1}\right)\left(1-q^{a(s)+1}t^{l(s)}\right)\\
&=\prod_{m=1}^{\gamma_1}\left(1-q^{k-r-m}t^{\gamma_m'+1}\right)\left(1-q^{k-r-m+1}t^{\gamma_m'}\right)
  \prod_{m=\gamma_1+1}^{k-r}\left(1-q^{k-r-m}t\right)\left(1-q^{k-r-m+1}\right)\\
&=q^{-2(k-r)(k-r-1)}\prod_{m=1}^{\gamma_1}\left(t^{\gamma_m'+1}q^{k-r-1}-q^{m-1}\right)\left(t^{\gamma_m'}q^{k-r}-q^{m-1}\right)\\
&\qquad\qquad\cdot\prod_{m=\gamma_1+1}^{k-r}\left(tq^{k-r-1}-q^{m-1}\right)\left(q^{k-r}-q^{m-1}\right).
\end{align*}
Thus,
\begin{align*}
\left(\frac{X_{(k)}^\lam}{X_{(r)}^\gamma}\right)^2\frac{c_\gamma c_\gamma'}{c_\lam c_\lam'}
&= q^{2(k-r)(k-r-1)}\prod_{m=1}^{\gamma_1}\frac{(t^{\gamma_m'}-q^{m-1})}{(t^{\gamma_m'+1}q^{k-r-1}-q^{m-1})}
  \frac{(t^{\gamma_m'}-q^{m-1})}{(t^{\gamma_m'}q^{k-r}-q^{m-1})}\\
&\qquad\qquad\cdot\prod_{m=\gamma_1+1}^{k-r}\frac{(1-q^{m-1})}{(tq^{k-r-1}-q^{m-1})}
 \frac{(1-q^{m-1})}{(q^{k-r}-q^{m-1})}.
\end{align*}
Since $k-r-1\geq m-1$ and $t>1$, then
$t^{\gamma_m'+1}q^{k-r-1}-q^{m-1}>0$, so that $q^{-(k-r-1)}t^{-1}<1$
implies
\begin{equation*}
\frac{(t^{\gamma_m'}-q^{m-1})}{(t^{\gamma_m'+1}q^{k-r-1}-q^{m-1})}<q^{-(k-r-1)}t^{-1}.
\end{equation*}
Similarly, since $k-r>m-1$ and $t>1$, then
$t^{\gamma_m'}q^{k-r}-q^{m-1}>0$, so that $q^{-(k-r)}< 1$ implies
\begin{equation*}
\frac{(t^{\gamma_m'}-q^{m-1})}{(t^{\gamma_m'}q^{k-r}-q^{m-1})}<q^{-(k-r)}.
\end{equation*}
Similarly, $t^{-1}q^{-(k-r-1)}$ and $q^{-(k-r)}<1$ imply
\begin{equation*}
\frac{(1-q^{m-1})}{(tq^{k-r-1}-q^{m-1})}<t^{-1}q^{-(k-r-1)}
\quad\text{and}\quad
\frac{(1-q^{m-1})}{(q^{k-r}-q^{m-1})}<q^{-(k-r)}.
\end{equation*}
So
\begin{align*}
\left(\frac{X_{(k)}^\lam}{X_{(r)}^\gamma}\right)^2\frac{c_\gamma c_\gamma'}{c_\lam c_\lam'}
&\leq q^{2(k-r)(k-r-1)}\prod_{m=1}^{\gamma_1} \left(q^{-(k-r-1)}t^{-1}\right)\left(q^{-(k-r)}\right)
 \prod_{m=\gamma_1+1}^{k-r}\left(t^{-1}q^{-(k-r-1)}\right)\left(q^{-(k-r)}\right)\\
&=q^{2(k-r)(k-r-1)}t^{-(k-r)}q^{-(k-r)^2}q^{-(k-r-1)(k-r)}= q^{-(k-r)}t^{-(k-r)}.
\end{align*}
Thus,
\begin{align*}
\frac{\barf_\lam(k)^2}{\barf_\gamma(r)^2}&=\left(\frac{X_{(k)}^\lam}{X_{(r)}^\gamma}\right)^2
\frac{c_\gamma c_\gamma'}{c_\lam c_\lam'}\frac{(q^k-1)^2}{(q^r-1)^2}\frac{(t,q)_k}{(q,q)_k}\frac{(q,q)_r}{(t,q)_r}\\
&\leq\frac1{q^{k-r}t^{k-r}}\frac{(q^k-1)^2}{(q^r-1)^2}\frac{(t,q)_k}{(q,q)_k}\frac{(q,q)_r}{(t,q)_r}. \qedhere
\end{align*}
\end{proof}

We may now bound the upper bound sum on the right hand side of
\eqref{52}.  Fix $j^*\leq k/2$. Using monotonicity (Lemma
\ref{lem53}), Lemma \ref{lem52}, Lemma \ref{lem51}, and the definition
of $\pi_{q,t}(r)$ from \eqref{12},
\begin{align*}
\sum_{\lam\ne (k)\atop \lam_1\ge k-j^*}\barf_\lam(k)^2\beta_\lam^{2\ell}&=\sum_{r=1}^{j^*} 
\sum_{\lam=(k-r,\gamma)}\beta_{(k-r,\gamma)}^{2\ell} \barf_\lam(k)^2 \leq \sum_{r=1}^{j^*} \sum_{\lam=(k-r,\gamma)} \beta_{(k-r,r)}^{2\ell} \barf_\lam(k)^2\\
&\leq\sum_{r=1}^{j^*} \beta_{(k-r,r)}^{2\ell} \sum_{\gamma \vdash r} \barf_\gamma(r)^2
\frac{(1-q^{-k})^2}{(1-q^{-r})^2} \frac{q^k}{t^k}\frac{(t,q)_k}{(q,q)_k}\frac{t^r}{q^r}\frac{(q,q)_r}{(t,q)_r} \\
&\leq\sum_{r=1}^{j^*}\beta_{(k-r,r)}^{2\ell}\frac1{\pi_{q,t}(r)}
 \frac{(1-q^{-k})^2}{(1-q^{-r})^2} \frac{q^k}{t^k}\frac{(t,q)_k}{(q,q)_k}
 \frac{t^r}{q^r}\frac{(q,q)_r}{(t,q)_r} \\
&\leq\sum_{r=1}^{j^*}\beta_{(k-r,r)}^{2\ell}r \frac{q^r}{t^r}\frac{(t,q)_r}{(q,q)_r}\frac{(1-q^{-r})}{(1-t^{-r})}
 \frac{(1-q^{-k})^2}{(1-q^{-r})^2} \frac{q^k}{t^k}\frac{(t,q)_k}{(q,q)_k}
 \frac{t^r}{q^r}\frac{(q,q)_r}{(t,q)_r} \\
&\leq\left(1-q^{-k}\right)^2 \frac{q^k}{t^k}\frac{(t,q)_k}{(q,q)_k}\sum_{r=1}^{j^*} r\beta_{(k-r,r)}^{2\ell}
 \frac1{(1-q^{-r})(1-t^{-r})} \\
&\leq\frac{(1-q^{-k})^2}{(1-q^{-1})(1-t^{-1})}\frac{q^k}{t^k}\frac{(t,q)_k}{(q,q)_k}\sum_{r=1}^{j^*} r\beta_{(k-r,r)}^{2\ell}.
\end{align*}
Using \eqref{Arunsbeta} and Lemma \ref{lem51} gives
\begin{align*}
\sum_{\lam\ne(k)\atop \lam_1\ge k-j^*} \barf_\lam(k)^2\beta_\lam^{2\ell}&\leq\frac{(1-q^{-k})}{(1-q^{-1})}
 \left(\prod_{j=1}^{k-1}\frac{1-t^{-1}q^{-j}}{1-q^{-j}}\right)
 \sum_{r=1}^{j^*}\frac{r}{q^{2r\ell}}\left(1+\frac1{tq^{k-2r}}\right)^{2\ell}\\
&\leq\frac{(1-q^{-k})}{(1-q^{-1})}\left(\prod_{j=1}^{\infty}\frac{1-t^{-1}q^{-j}}{1-q^{-j}}\right)
 \left(1+\frac1{tq^{k-2j^*}}\right)^{2\ell}\sum_{r=1}^{j^*}\frac{r}{q^{2r\ell}}\\
&\leq\frac{(1-q^{-k})}{(1-q^{-1})^{3/2}}
 \left(1+\frac1{tq^{k-2j^*}}\right)^{2\ell}\frac1{q^{2\ell}}\left(1-\frac{1}{q^{2\ell}}\right)^{-2}\\
&\leq\frac{(1-q^{-k})}{(1-q^{-1})^{3/2}}\frac1{(1-q^{-2})^2}
 \left(\frac1{q}+\frac1{tq^{k-2j^*+1}}\right)^{2\ell},
\end{align*}
by Lemma \ref{lem52}. Choose $j^*$ (of order $k/4$) so that $k-2j^*+1
= k/2$. Then
\begin{align*}
\sum_{\lam\ne(k)\atop \lam_1\ge k-j^*} \barf_\lam(k)^2\beta_\lam^{2\ell}\leq 
 \frac1{(1-q^{-1})^{3/2}(1-q^{-2})^2}\left(\frac1{q}+\frac1{tq^{k/2}}\right)^{2\ell},
\end{align*}
with $a$ as in the statement of Theorem \ref{thm51}.

Now use
\begin{align*}
\sum_{\lam\atop \lam_1< 3k/4} \barf_\lam(k)^2\beta_\lam^{2\ell}&\leq 
\sum_{\lam\atop \lam_1< 3k/4} \barf_\lam(k)^2\beta_{(\frac{3k}{4},\frac{k}{4})}^{2\ell}\\
&\leq\sum_{\lam} \barf_\lam(k)^2\beta_{(\frac{3k}{4},\frac{k}{4})}^{2\ell} 
 \leq\frac1{\pi_{q,t}(k)}\beta_{(\frac{3k}{4},\frac{k}{4})}^{2\ell} \\
&=\frac{t^k}{q^k}\frac{(q,q)_k}{(t,q)_k}\frac{(1-t^{-k})}{(1-q^{-k})}k \beta_{(\frac{3k}{4},\frac{k}{4})}^{2\ell}\\ 
&\leq k \frac{(1-t^{-k})}{(1-t^{-1})}
\left(\prod_{j=1}^{k-1} \frac{1-q^{-j}}{1-t^{-1}q^{-j}}\right)
\left(\frac{1}{q^{k/4}}\big( 1+ \frac{1}{tq^{k/2}}\big)\right)^{2\ell} 
\end{align*}
so that
\begin{equation}
\sum_{\lam\atop \lam_1< 3k/4} \barf_\lam(k)^2\beta_\lam^{2\ell}\leq 
k\frac{t}{t-1}\left(\frac2{q^{k/4}}\right)^{2\ell}.
\label{516}
\end{equation}
This completes the proof of Theorem \ref{thm51}. \qed

\subsection{Bounds starting at $(1^k)$}\label{sec54}

We have not worked as seriously at bounding the chain starting from
the partition $(1^k)$. The following results show that $\log_q(k)$
steps are required, and offer evidence for the conjecture that
$\log_q(k)+\theta$ steps suffice (where the distance to stationarity
tends to zero with $\theta$, so there is a sharp cutoff at
$\log_q(k)$).

The $L^2$ or chi-square distance on the right hand side of \eqref{52}
has first term $\beta_{k-1,1}^{2\ell}\barf_{k-1,1}^2(1^k)$.
\begin{lemma}
  For fixed $q,t>1$, as $k$ tends to infinity, 
\begin{equation*}
\barf_{(k-1,1)}(1^k)^2=\frac{\left(X_{(1^k)}^{(k-1,1)}(q-1)^k\right)^2}
{c_{(k-1,1)}c_{(k-1,1)}'(q,q)_k/(t,q)_k}\sim\left(\frac{1-q^{-1}}{1-t^{-1}}\right)k^2.
\end{equation*}
\label{lem57}
\end{lemma}
\begin{proof}
  From \eqref{33} and the definition of $\varphi_T(q,t)$ from \cite[VI
  p.~341 (1)]{mac} and \cite[VI (7.11)]{mac},
\begin{align*}
X_{(1^k)}^{(k-1,1)}=\frac{c_{(k-1,1)}'(q,t)}{(1-t)^k}\sum_T\varphi_T(q,t) 
=\frac{c_{(k-1,1)}'}{(1-t)^k}\left(\frac{(1-t)^k}{(1-q)^k}p\right)
=\frac{c_{(k-1,1)}'}{(1-q)^k} p,
\end{align*}
with
\begin{equation*}
p=\left(\frac{\frac{1-t^2}{1-qt}}{\frac{1-t}{1-q}}+
\frac{\frac{1-qt^2}{1-q^2t}}{\frac{1-qt}{1-q^2}}+
\frac{\frac{1-q^2t^2}{1-q^3t}}{\frac{1-q^2t}{1-q^3}}+\dots+
\frac{\frac{1-q^{k-2}t^2}{1-q^{k-1}t}}{\frac{1-q^{k-2}t}{1-q^{k-1}}}\right).
\end{equation*}
Using the definition of $c_{(k-1,1)}$ and $c_{(k-1,1)}'$ from \eqref{32},
and the definition of $(t,q)_k$ and $(q,q)_k$ from \eqref{12},
\begin{align*}
\barf_{(k-1,1)}(1^k)^2&=\frac{\left( X_{(1^k)}^{(k-1,1)} (1-q)^k\right)^2 (t,q)_k}
{c_{(k-1,1)}c_{(k-1,1)}'(q,q)_k}=\frac{c_{(k-1,1)}'p^2}{c_{(k-1,1)}}\frac{(t,q)_k}{(q,q)_k}\\
&= \frac{(t,q)_k}{(q,q)_k}
\frac{(1-q)(1-tq^{k-1})(1-q)(1-q^2)\cdots (1-q^{k-2})}{(1-t)(1-t^2q^{k-2})(1-t)(1-tq)\cdots (1-tq^{k-3})} p^2 \\
&= \frac{ (1-tq^{k-2})(1-tq^{k-1})}{(1-q^{k-1})(1-q^k)} 
\frac{(1-q)(1-tq^{k-1})}{(1-t)(1-t^2q^{k-2})} p^2 \\
&= \frac{ (1-t^{-1}q^{-(k-2)})(1-t^{-1}q^{-(k-1)})}{(1-q^{-(k-1)})(1-q^{-k})} 
\frac{(1-q^{-1})(1-t^{-1}q^{-(k-1)})}{(1-t^{-1})(1-t^{-2}q^{-(k-2)})} p^2,
\end{align*}
and the result follows, since $p\sim k$ for $k$ large.
\end{proof}
\begin{cor}
  There is a constant $c$ such that, for all $k,\ell\geq2$,
\begin{equation*}
\chi_{(1^k)}^2(\ell)=\sum_\lam\frac{\left(M^\ell((1^k),\lam)-\piqt\lamp\right)^2}{\piqt\lamp}
\geq \left(\frac{1-q^{-1}}{1-t^{-1}}\right)\frac{k^2}{q^{2\ell}}.
\end{equation*}
\label{cor52}
\end{cor}
\begin{proof}
  Using only the lead term in the expression for
  $\chi_{(1^k)}^2(\ell)$ in \eqref{54} gives the lower bound
  $\beta_{(k-1,1)}^{2\ell}\barf_{(k-1,1)}^2(1^k)$. The formula for
  $\beta_{(k-1,1)}$ in Theorem \ref{thm2}(2) gives
  $\beta_{(k-1,1)}\geq\frac1{q}$, and the result then follows from
  Lemma \ref{lem57}.
\end{proof}

The corollary shows that if $\ell=\log_q(k)+\theta,\
\chi_{1^k}^2(\ell)\geq\frac{c}{q^{2\theta}}$. Thus, more than
$\log_q(k)$ steps are required to drive the chi-square distance to
zero. In many examples, the asymptotics of the lead term in the bound
\eqref{52} sharply controls the behavior of total variation and
chi-square convergence. We conjecture this is the case here, and that
there is a sharp cut-off at $\log_q(k)$.

It is easy to give a total variation lower bound:
\begin{prop}
  For the auxiliary variables chain $M(\lam,\lam')$, after $\ell$
  steps with $\ell=\log_q(k)+\theta$, for $k$ large and
  $\theta<-\frac{t-1}{q-1}$,
\begin{equation*}
\left\|M_{(1^k)}^\ell-\piqt\right\|_{\text{TV}}\geq e^{-\frac{t-1}{q-1}}-e^{-\frac1{q^\theta}}+o(1).
\end{equation*}
\label{prop51}
\end{prop}
\begin{proof}
  Consider the Markov chain starting from $\lam=(1^k)$. At each stage,
  the algorithm chooses some parts of the current partition to
  discard, with probability given by \eqref{13}. From the detailed
  description given in \ref{sub2d3}, the chance of a specific
  singleton being eliminated is $1/q$. Of course, in the replacement
  stage \eqref{14} this (and more singletons) may reappear. Let $T$ be
  the first time that all of the original singletons have been removed
  at least once; this $T$ depends on the history of the entire Markov
  chain. Then $T$ is distributed as the maximum of $k$ independent
  geometric random variables $\{X_i\}_{i=1}^k$ with
  $P(X_i>\ell)=1/q^\ell$ (here $X_i$ is the first time that the $i$th
  singleton is removed).

  Let $A=\{\lam\in\calp_k:a_1\lamp>0\}$. From the definition
  \begin{equation*}
    \left\|M_{(1^k)}^\ell-\piqt\right\|_{\text{TV}}
    =\max_{B\subseteq\calp_k}\left|M^\ell\left((1^k),B\right)-\piqt(B)\right|
    \geq\left|M^\ell\left((1^k),A\right)-\piqt(A)\right|
  \end{equation*}
  and
  \begin{align*}
    M^\ell\left((1^k),A\right)\geq P\{T>\ell\}&=1-P\{T\leq\ell\}\\ 
    &=1-P\{\max X_i\leq\ell\}\\ 
    &=1-P(X_1\leq\ell)^k\\
    &=1-e^{k\log\left(1-P(X_1>\ell)\right)}\\
    &=1-e^{k\log(1-1/q^\ell)}\sim 1-e^{-k/q^\ell}\\ &=1-e^{-1/q^\theta}.
  \end{align*}
  From the limiting results in \ref{sub2d5}, under $\piqt,\ a_1\lamp$
  has an approximate Poisson $\left(\frac{t-1}{q-1}\right)$
  distribution. Thus, $\piqt(A)\sim1-e^{-\frac{t-1}{q-1}}$. The result
  follows.
\end{proof}

\section*{Acknowledgments}

We thank John Jiang and James Zhao for their extensive help in
understanding the measures $\piqt$. We thank Alexei Borodin for
telling us about multiplicative measures and Martha Yip for help and
company in working out exercises from \cite{mac} over the past several
years.  We thank Cindy Kirby for her expert help.

\bibliography{macdonald}

\end{document}